\documentclass[10pt,a4paper]{amsart}

\usepackage{amssymb,amsmath,graphicx,amsfonts,euscript}
\usepackage{color}

\newtheorem{thm}{Theorem}[section]

\newtheorem{rem}[thm]{Remark}

\newtheorem{lemma}[thm]{Lemma}

\newcommand{\p}{\partial}

\newcommand{\De}{\Delta}

\newcommand{\R}{\mathbb{R}}
\newcommand{\Z}{\mathbb{Z}}

\newcommand{\h}{\mathcal{H}}

\newcommand{\lp}{\mathfrak{l}}

\newcommand{\kk}{\mathfrak{k}}
\newcommand{\N}{\mathbb{N}}
\numberwithin{equation}{section}
\begin{document}
\title[Hilbert transform and maximal operator]{$L^p$ estimates for Hilbert transform and maximal operator  associated to   variable polynomial}

\author[R. Wan]{ Renhui Wan}

\address{School of Mathematical Sciences, Nanjing Normal University, Nanjing 210023, People's Republic of China}

\email{rhwanmath@163.com}

\vskip .2in
\begin{abstract}
We investigate the  Hilbert transform and the maximal operator  along a  class of variable non-flat polynomial curves $(P(t),u(x)t)$ with measurable $u(x)$,  and prove uniform $L^p$ estimates  for $1<p<\infty$.  In particular,  via the change of variable, these uniform estimates are equal to  the  ones  for the curves  $(P(v(x)t),t)$ with  measurable $v(x)$. To obtain the desired bound, we make full use of time-frequency techniques and establish a crucial $\epsilon$-improving estimate  for some special separate sets.

\end{abstract}

\maketitle
\section{Introduction}
\label{s1}
In the present paper, we investigate the  Hilbert transform and the maximal operator  along  variable  non-flat polynomial curves in $\R^2$ described  by
\begin{equation}\label{curve}
\Gamma_x(t):=(P(t),u(x)t),\ \ {\rm where}\ \ P(t)=\sum_{i=2}^Na_it^i,\ \{a_i\}_{i=2}^N\subset \R,
\end{equation}
where the function $u(x):\R\mapsto \R$  is measurable,
and obtain  the uniform $L^p$ estimates for $1<p<\infty$. More precisely,  we study 
the Hilbert transform along $\Gamma_x(t)$ defined by
\begin{equation}\label{hhh}
\h^\Gamma f(x,y):=p.v.\int_\R f((x,y)-\Gamma_x(t))\frac{dt}{t}
\end{equation}
and 
the maximal operator along $\Gamma_x(t)$  given by
\begin{equation}\label{mmm}
M^\Gamma f(x,y):=\sup_{\epsilon>0}\frac{1}{2\epsilon}
\int_{-\epsilon}^\epsilon |f((x,y)-\Gamma_x(t))|dt.
\end{equation}
\vskip.1in
We state our main result as follows:
\begin{thm}\label{t1}
Let $u(x)$  and
$\Gamma_x(t)$ be given by (\ref{curve}).  Then $\h^\Gamma $ and  $M^\Gamma $ defined as in (\ref{hhh}) and (\ref{mmm}) can be extended to two bounded operators from $L^p(\R^2)$ to $L^p(\R^2)$
for $1<p<\infty$.  In addition, the bounds are uniform in the  sense of  that  they  depend only on the degree $N$.
\end{thm}
\begin{rem}\label{r1}
%(1)  We only give the detailed proof of the  $L^p$ estimate  of the more involved  $%\h^\Gamma$, and choose to provide a brief discussion  for $M^\Gamma$ (see section \ref{me}). We further remark that the estimate 
  %is a direct result of the corresponding  work in   \cite{CRW98}  when $u(x)=0$.\\
The curve $\Gamma_x(t)$ in (\ref{curve}) can be generalized to 
$(\sum_{i=1}^N a_i [t]^{\beta_i}$, $u(x)[t]^\alpha)$, where $\alpha>0$,
 $a_i\in\R,\ \beta_i\neq \alpha$ and $\beta_i>0.$ Here $[t]^\sigma=|t|^\sigma$ or $sgn(t)|t|^\sigma$.
 In order not to  clutter the presentation, we do not
choose  to pursue on this direction here.
%We use a unified way to bound the above three operator
%(3) We only need to consider  $u(x)\neq 0$ below since
% $u(x)=0$ is a direct result of certain main theorem in   %arbery, Ricci and Wright
% \cite{CRW98}.
Furthermore, 
the arguments in the proof of Theorem \ref{t1}  also work  for the Carleson type operator 
$p.v.\int f(x-P(t))e^{iu(x) t}\frac{dt}{t}$,  see   section \ref{me} for its discussion.
\end{rem}
\subsection{Historical background}
 Replacing $\Gamma_x (t)$   in (\ref{curve}) by $(t,u(x,y)t)$  and the  domain of integration in (\ref{hhh}) by  $[-\epsilon_0,\epsilon_0]$, we reduce the operators $M^\Gamma$ and $\h^\Gamma$   to the $\epsilon_0$-truncated operators $M^\Gamma_{\epsilon_0}$ and $\h^\Gamma_{\epsilon_0}$, which are
  the  objects investigated in the so-called Zygmund conjecture and Stein conjecture. For these conjectures, the  zero-curvature case as well as the function $u(x,y)$,  which   depends on the double  variables and is assumed only in the $Lip$ space,   make them so difficult that both are open so far. Nonetheless, there are various variants of them.  Here we list some partial progresses connected with the present work as follows.
  %\vskip.1in
  %We first show the historical background of
  %\h^\Gamma$ and $M^\Gamma$.
  \vskip.1in
  \underline{The zero-curvature case}\ \ \
 Bourgain  \cite{B89} proved  that $M^\Gamma_{\epsilon_0}$ with any real analytic  $u=u(x,y)$ is $L^2$ bounded.  The analogous  result  for  $\h^\Gamma_{\epsilon_0}$ was proved  later by Stein and Street  \cite{SS12},  whose 
   objects are all polynomials with analytic coefficients. For smooth $u=u(x,y)$ satisfying  certain curvature condition, Christ et al. \cite{CNSW99} demonstrated that  both $\h^\Gamma$ and $M^\Gamma$ are $L^p$ bounded.
 Later,    Lacey and Li \cite{LL10} obtained   by a sophisticated  time-frequency approach  exploited by them  in \cite{LL06}  that $\h^\Gamma_{\epsilon_0}$ with $u\in \mathcal{C}^{1+\epsilon}$ %(H\"{o}lder space) 
  is bounded on $L^2$.  
 %It is worth remarking that \cite{LL06}  made
%  the first major breakthrough in terms of regularity of $u(x,y)$. 
   For any measurable  $u=u(x)$,
  Bateman \cite{B13} and Bateman-Thiele \cite{BT} proved that $\h^\Gamma $ is $L^p$ bounded   for $p>3/2$ and   $\h^\Gamma P_k^{(2)}$ is bounded on $L^p$ for $p>1$, where $P_k^{(2)}$ denotes the Littlewood-Paley projection in the $y$-variable and  the commutation relation $\h^\Gamma P_k^{(2)}=P_k^{(2)}\h^\Gamma $ is  crucial in their proofs.
  Recently, partially motivated by Bateman'work \cite{B13},
  via establishing a crucial $L^p$ estimate of certain commutator,
  Guo \cite{G17} proved  $\h^\Gamma$ is $L^p$ bounded for $p>3/2$ under the assumption that   $u=u(h(x,y))$ with  sufficiently small $\|\nabla h -(1,0)\|_\infty$.    If the measurable  $u=u(x,y)$ does not own any regularity,  $\h^\Gamma$ along  the curve $(t,u(x,y)t)$ is not bounded on  $L^p$,  see \cite{K07,LMP19}.
  %where $u(h(x,y))$
  %is consistent with
  %u(x)$ if we set $h(x,y)=x$.
  \vskip.1in
   \underline{The non-zero  curvature case}\ \ \
   This  problem is not only a non-trivial generalization of the zero curvature but also closely related to the Carleson maximal operators. Here and in what follows   $P_x(t)$ ($P_y(t),P_{x,y}(t)$) denotes the polynomial with the coefficients depending on the $x$-variable ($y$-variable, both $x$-variable and $y$-variable).
   %is consistent with the $L^2$ estimate of the %associated Carleson maximal function by taking Fourier %ransform in $y-$variable.
   For $\Gamma_x(t)=(t,u(x,y)[t]^\alpha)$ $(0<\alpha\neq 1)$ with  measurable function $u(x,y)$, Marletta and Ricci \cite{MR98} obtained $M^\Gamma$ is $L^p$ bounded for  $p>2$.
    For  $u(x,y)\in {\rm Lip}(\R^2)$,  Guo et al. \cite{GHLJ}  proved   $M^\Gamma_{\epsilon_0}$ with certain $\epsilon_0=\epsilon_0(\|u\|_{\rm Lip})$ is bounded on $L^p$ for $p\in(1,2]$.
   Furthermore, Guo et al.  proved 
   that $\h^\Gamma P_k^{(2)}$ with measurable $u=u(x,y)$ is $L^p$ bounded for $p\in (2,\infty)$. 
   Under the assumption that  $\|u\|_{\rm Lip}$ is small enough, Di Plinio et al. \cite{DGTZ} obtained $\h^\Gamma_{1}$ is bounded on $L^p$ for $p\in (1,\infty)$.
   Very recently, Liu-Song-Yu \cite{LSY21} and Liu-Yu \cite{LY22} used local smoothing estimates for Fourier integral operators  to extend \cite{MR98,GHLJ,DGTZ} to a larger class of curves $(t,u(x,y)\gamma(t))$, where  $\gamma(t)$
   is even or odd.  However, it seems difficult to use the ideas in  the above works to accomplish  the uniform estimate for the variable curve $(t,P_{x}(t))$ with  $P_{x}'(0)=0$ and measurable coefficients, where the uniform estimate is in the  sense of  that its
    bound depends only on the degree $N$. For this curve,
    Lie \cite{L19} recently  used a unified approach to obtain the uniform  $L^p$-boundedness $(1<p<\infty)$ of $M^\Gamma$, $\h^\Gamma$ and related operators by the LGC-methodology (see page 9 in that paper for the details).
    In fact, some more general cases are proved in that paper. For the special case $\Gamma_x(t)=(t,a_2(x)t^2+a_3(x)t^3)$,  via
     making full use of Littlewood-Paley theory and the  commutation relation $\h^\Gamma P_k^{(2)}=P_k^{(2)}\h^\Gamma$, Wan  \cite{Wan19} proved $\h^\Gamma$ is $L^p$ bounded for  $p\in(1,\infty)$. 
     For $\Gamma_x(t)=(t,u(x,y)[t]^b)$ with $b>1$,    if the measurable  $u=u(x,y)$ does not own any regularity,
     Guo et al.  \cite{GRSP,GRSP2} showed that   $\h^\Gamma$ is not bounded on $L^p$ for any   $p\in(1,\infty)$, which  is different from the operator $M^\Gamma$, we refer \cite{MR98}.  
   % \vskip.1in
   %We now give the historical background of %$\mathcal{C}^\Gamma$.
   %\vskip.1in
   %When we replace $(P(t),u(x)t)$ by $(t,P_x(t))$, where
   %P_x(t)=\sum_{i=1}^N a_i(t)t^i$, it reduces to the %one-dimensional version of another conjecture proposed %by Stein
\vskip.1in
\subsection{ Motivations}
The uniform $L^p$ estimates of the linear and multilinear singular integral operators and the maximal operators along various  variable polynomial curves  have been studied,   
however, for $\h^\Gamma$ and $M^\Gamma$ along the curve $(t,P_y(t))$ with $P_{y}'(0)=0$  or $(P_x(t),t)$ with $P_{x}'(0)=0$, we do not know so far whether their uniform estimates hold without imposing  any  assumptions to the coefficients. This paper gives partial progress on this question.  More precisely, 
we   investigate  the uniform $L^p$ estimates of $\h^\Gamma$ and $M^\Gamma$ along  a new class of variable curves $(P(t),u(x)t)$ with $P'(0)=0$. Indeed, the change of variable gives that each curve can be transformed to the one like $(P(v(x)t),t)$ with certain measurable $v(x)$, which is  a special case of $(P_x(t),t)$ and 
can not be treated  by directly  using the arguments in the previous works. In addition, 
 we only pay attention to the polynomial whose coefficients do not depend on the variable(s) since $\h^\Gamma$ may not be  bounded otherwise, see \cite{GHLJ}.
  %As a matter of fact, if  the  coefficients depend only on the $y-$variable, changing the %variable $t\rightarrow u(x)^{-1}t$ when $u(x)\neq 0$, we  reduces $\h^\Gamma$ to
% $p.v. \int_\R f(x-P_y(u(x)^{-1}t),y-t)\frac{dt}{t},$
% which is not  bounded by applying the theorem in \cite{GHLJ}. And if  the  coefficients depend %only on the $x-$variable, motivated by the previous works which needed the $Lip$ condition, %it seems necessary to require at least that the  coefficients  of $P_x(t)$ is in $Lip$ space in %order  to achieve the desired uniform estimate.
% Another motivation is the work  \cite{CRW98}, where the authors investigated the lower %dimensional version of $\h^\Gamma$ and $M^\Gamma$, namely, $p.v.\int_\R f(x-P(t))\frac{dt}%{t}$ and $\sup_{\epsilon>0}\frac{1}{2\epsilon}
%\int_{-\epsilon}^\epsilon |f(x-P(t))|dt$. Although they considered all polynomials including linear %term, one can see the appearance of $u(x)t$ in the present paper makes the problem harder.

\subsection{Outline of the proof and comments}
 Since the present work belongs to the non-zero curvature cases, some arguments  are   related  to those of the previous works  such as \cite{GHLJ,L19,LX16}.
  By a decomposition technique, we first  reduce the goal to the  estimates   related to   $N-1$  ``dominating sets".  
In every ``dominating set", we reduce
the objective estimate   into four  segments via the partition of unity given by (\ref{fenjie}) in order to %$M_j(\xi,\eta)$, where
%$\De_\sigma \subset \{(m,n)\in \Z^2\}$ is defined by %(\ref{dde}), and  $2^m\thicksim$ the absolute value of   the %first part of the phase function while $2^n$ approximately %equals to the absolute value of the second part of the phase %function.
 to
 quantify   the phase function more conveniently. We list the approaches treating each segment
 as follows:
 %, while
 %this partition is involved with  both frequency space  and  physical space. 
\vskip.1in
1 %The choice of $\De_\sigma$ is mainly based on the %features of the multiplier, an oscillatory integral.
The first segment is  the low frequency case which is estimated by Taylor's expansion, while the second segment is the off-diagonal frequency case which we deal with by exploiting
 integration by parts as well as  Taylor's expansion.
  We remark  that  the function of Taylor's expansion in the second segments is  to exploit the large lower bound of the derivative of the phase function.
%We need Abel summation to control $\sum_{j\in S_l}$ in the %estimate  corresponding to $\De_1$,
 In addition, to estimate  the second segment, we also need
 the vector-valued shifted maximal estimate and the estimates of some variants of
 vector-valued singular    integrals  (which are also used in the fourth segment).   For the proof of the third segment which belongs to the off-diagonal frequency case as well, we only give a sketch since it can be handled by  combining     the previous two approaches treating   the first and the second segments.
\vskip.1in
2 The estimate  of the fourth segment is obtained by establishing 
a new  $\epsilon$-improving estimate (see Lemma \ref{l31}) and developing some important 
arguments  in  \cite{L19}.  Indeed,  the proof in the current paper is  
 more involved than  the previous works.
%Our strategy is reducing the problem to the estimate of %$\mathcal{H}_{\De_4,m}f$ (see section \ref{hh}),
%establishing the $L^2$ estimate with exponential decay $2^{-\epsilon_0m}$ and the $L^p$ estimate with monomial growth $m^2$, and using interpolation in the end.
%Although the main structure has been applied in many  works,
%there are some new ingredients   in our proof.
The main novelty is  to achieve  the exponential decay for the $L^2$ estimate of $\mathcal{H}_{\De_4,m}f$ (see section \ref{hh}),  which is the key part in this paper.
%where establishing $L^2$ estimate is the toughest.
  % Because of the appearance of a critical point of the phase function, the strategies  treating the previous cases do not work.  Fortunately,  
 The  non-degenerate phase function makes us  apply the method of stationary phase to get an asymptotics. However, because of   the phase function depending on  $x$ variable,  we need several new ideas in the  estimate of  this asymptotics. More precisely, these ideas are included in the following  steps: 
   \begin{itemize}
   \item
   discretizing the phase function and  reducing the problem to  the analysis of the integrand in certain new  oscillatory integral; 
   \vskip.1in
    \item
    expressing  this  integrand by regarding it  as a periodic function  and  making full use of  $TT^\star$ method, and then reducing  the goal to the estimate of  certain integral  expected to own an exponential
   decay; 
      \vskip.1in
    \item
   establishing a crucial $\epsilon$-improving estimate given by Lemma \ref{l31}  for certain  ``bad" set and then verifying the decay estimate in the former step. 
   \end{itemize}
   %  \vskip.1in
   %  The above three steps  are   achieved by  establishing an $\epsilon$-improving estimate 
     
   %  developing certain important arguments in  \cite{L19} and  establishing  Lemma \ref{p1}, %which works for the polynomial compounded  with certain  functions and plays
  %  a crucial role in our paper.
  %    The second is that we prefer considering the above   integrand as a periodic function %which seems natural for the readers rather than using  the essential Garbor-frame in \cite{L19}. 

    %\begin{rem}\label{com}
   % We point out  that the idea of, two more general  results %are obtained in the present paper, that is, Lemma \ref{l31} %and Lemma \ref{p1}, which are  of independent interesting.
    %In fact,
    %  when $F(\tau)=\tau$, they reduce to the objective %investigated in \cite{L19}.
      %In addition, the generalization of  Lemma \ref{p1} %is so difficult that it seems hard to extend the %polynomial in $\Gamma(x,t)$ to general function.
    %\end{rem}
\vskip.1in
{\bf Organization of the paper}
 In Section \ref{s2},  we give the partition of unity and the reduction of Theorem \ref{t1}.  The third section lists two auxiliary consequences which are applied to the tricky  estimate in the sixth section, and the followed section  gives  the proof of Theorem \ref{t2} for $\sigma=1$.  In the fifth section, we prove Theorem \ref{t2} for $\sigma=2$ and $\sigma=3$. In the sixth section, Theorem \ref{t2} for $\sigma=4$ is proved by the auxiliary consequences in the third section. In the 7th section, we give the  proof of Lemma \ref{l6.1} which is used in the sixth section. At last, we recall some useful results including the shifted maximal operator in the Appendix.
 \vskip.1in
{\bf Notations}\ \ 
We use $e(x)=e^{2\pi i x}$. The Fourier transform $\widehat{f}(\xi)$ of $f(x)$ is defined by $\int e(-\xi x) f(x) dx$, while $g^{\vee}$ is the Fourier inverse transform of $g$ defined by $g^{\vee}(x)=\int g(\xi) e(\xi x)d\xi$. For convenience,  hereinafter, we  omit $2\pi$ in the notation of $e(x)$.  $\mathcal{F}^y$ is the Fourier transform in the $y-$variable.
We   use $x\lesssim y$ to stand for there exists a uniform constant $C$  such that $x\le Cy$.
$x\gtrsim y$ means $y\lesssim x$, and  the notation $x\thicksim y$ signifies that $x\gtrsim y$ and $x\lesssim y$. The absolute or uniform constant in what follows may be hidden in ``$\lesssim$". We use
 $C_{\gamma},C(\gamma)$ to represent  the constants depending on $\gamma$, and the constants hidden in $\lesssim_N$ depend only on $N$.
 We use
$\|\cdot\|_p$ to stand for $\|\cdot\|_{L^p}$.
\section{The reduction of Theorem \ref{t1}}
\label{s2}
Let $\theta_+(t)$ be supported in $(\frac{1}{9},9)$ such that $\sum_{j\in\Z}\theta_+(2^jt)=1$ for all $t>0$. Let $\theta_-(t)=\theta_+(-t)$, $\theta(t)=\theta_+(t)+\theta_-(t)$, then
$\sum_{j\in\Z} \theta(2^jt)=1$ for all $t\neq0$.
Let $\rho(t)=\frac{\theta(t)}{t}$, denote $\rho_j(t):=2^j\rho(2^jt)$,
we have for all $t\neq 0$,
\begin{equation}\label{idd}
\frac{1}{t}=\sum_{j\in\Z}2^j\frac{\theta(2^jt)}{2^jt}
=\sum_{j\in\Z}\rho_j(t).
\end{equation}
%Before giving  the reduction, we first show some intuitions on this problem. If replacing the polynomial $P(t)$ in (\ref{curve}) by some special monomials like $a_3t^3$, it reduces to  the associated object in \cite{GHLJ} by  changing the variable $a_3t^3\rightarrow t$. However, this process does not work for the polynomial $P(t)$ even with the techniques below since it is so hard to establish an  oscillatory estimate like Lemma 2.2 in \cite{GHLJ}. In addition, this reason also yields one can not use the  commutation relation $\h^\Gamma P_k^{(2)}=P_k^{(2)}\h^\Gamma $ to simplify the proof in the present work.
\subsection{First reduction of Theorem \ref{t1}}
We only give a detailed proof  for  $\h^\Gamma f$ since $M^\Gamma f$ can be similarly treated. 
Setting  $\h_jf(x,y)$ as
$$\h_jf(x,y):=\int f(x-P(2^{-j}t),y-u(x)2^{-j}t)\rho(t)dt,$$
satisfying  $\|\h_jf\|_p\lesssim \|f\|_p$,
we obtain via (\ref{idd}) that
$
\h^\Gamma f(x,y)
=\ \sum_{j\in\Z}\h_jf(x,y).
$
Next,
we seek the ``dominating monomial" of the 
polynomial $P(2^{-j}t)$.
\vskip.1in
Now, we give a further  decomposition of $\h^\Gamma f$.
For $l=2,\cdot\cdot\cdot,N$, we denote
\begin{equation}\label{sing}
S_l=\big\{j\in\Z:\ |j|>2^N,\ |a_l|2^{-jl}>\digamma_N|a_i|2^{-ji}\ {\rm for\ all\ }\ i\neq l \ {\rm and }\ 2\le i\le N\big\},
\end{equation}
where $\digamma_N$ is a  large enough constant depending only on $N$.
It is easy to see $S_l\cap S_{l'}=\varnothing$ for $l\neq l'$.
Denote
$S_o=\big(\bigcup_{l=2}^N S_l\big)^c,$
then we have a decomposition of $\Z$, that is
$\Z=S_o\cup\big(\bigcup_{l=2}^N S_l\big)$.
In addition, for any $j\in S_o$, we have
$|j|\le 2^N$ or
there exists $(\lp,\lp')\in \{2,\cdot\cdot\cdot,N\}^2$ satisfying $\lp\neq \lp'$ such that
$|a_\lp|2^{-j\lp}\le C_1(N)|a_{\lp'}|2^{-j\lp'}\le C_2(N)|a_\lp|2^{-j\lp}$
for certain constants $C_2(N)\ge C_1(N)>0$,
which immediately yields
$\sharp S_o\lesssim_N1$
and
$\|\sum_{j\in S_o}\h_jf \|_p\lesssim_N\sup_{j\in S_o}\|\h_jf\|_p\lesssim_N \|f\|_p.$
Thus it suffices to show that $\|\sum_{j\in S_l}\h_jf \|_p\lesssim_N\|f\|_p$ holds for each $l\in \{2,\cdot,\cdot,\cdot,N\}$.
\subsection{Second reduction of Theorem \ref{t1}}
Via rescaling arguments, we can assume 
  \begin{equation}\label{p}
P(t)=t^l+\sum_{2=i\neq l}^N a_i t^i.
\end{equation}
In order  to avoid the negative effect of the coefficient of $t^l$, this  process is necessary.
Fourier inverse transform gives
$$\h_jf(x,y)=\int_{\xi,\eta}\widehat{f}(\xi,\eta)
e(\xi x+\eta y) M_j(\xi,\eta)d\xi d\eta,$$
where
$$M_j(\xi,\eta):=\int e(\phi_{j,\xi,\eta,x}(t))\rho(t)dt,\ \phi_{j,\xi,\eta,x}(t):=\xi P(2^{-j}t)+\eta u(x)2^{-j}t.$$
In the following, we will decompose the multiplier $M_j(\xi,\eta)$ into four main parts. Before we go ahead, we introduce the partition of unity as follows:
\begin{equation}\label{fenjie}
\sum_{(m,n)\in \Z^2}\sum_{k\in\Z}\widehat{\Phi}(\frac{\xi}{2^{jl+m}})
\widehat{\Phi}(\frac{\eta}{2^{k}})
\widehat{\Phi}(\frac{u(x)}{2^{j-k+n}})=1,
\end{equation}
where $\widehat{\Phi}(\cdot)=\widehat{\Phi}_+({\cdot})
+\widehat{\Phi}_-(\cdot)$,
and  $\widehat{\Phi}_\pm$ supported in $\pm[1/2,2]$ are defined as $\theta_\pm$. (\ref{fenjie}) is based on the phase function $\phi_{j,\xi,\eta,x}(t)$ in $M_j(\xi,\eta)$.
%$\widehat{\Phi}(\xi)$ supported in $\frac{1}{2},2)$ is %defined by a similar way yielding the definition of %$\theta$ at the beginning of  section \ref{s2}.
%Õâ±ßÆäÊµ¿ÉÒÔ¸Ä³É·ÇÆë´Î·Ö½â
To discuss the different behaviors of  $\phi_{j,\xi,\eta,x}(t)$, we also need  a  decomposition of  $\Z^2$: $\Z^2=\cup_{i=1}^4\De_i$ given by
\begin{equation}\label{dde}
\begin{aligned}
\Delta_1:=&\ \{(m,n)\in\Z^2:\ \max\{m,n\}\le 0\},\\
\De_2:=&\ \{(m,n)\in\Z^2: \ \max\{m,n\}>0,\ |m-n|> 100l,\ \min\{m,n\}>0
\}\\
\De_3:=&\ \{(m,n)\in\Z^2:\ \max\{m,n\}>0,\ |m-n|> 100l,\ \min\{m,n\}\le0\},\\
\De_4:=&\ \{(m,n)\in\Z^2:\ \max\{m,n\}>0,\ |m-n|\le  100l\}.
\end{aligned}
\end{equation}
We will give an explanation of this process  at the end of this section.
%Here $C_0$ is large enough (1000 is enough).
%Since the cancelation property of $\int \rho=0$, $\xi=0$ %and $\eta=0$ can not occur at the same time. In fact, %$(\xi,\eta)=(0,\eta)$ ($\eta\neq0$) (or the symmetry %case) can be bounded by the same strategy dealing with %the case $\xi\eta\neq0$. Together with $u(x)\neq 0$, so %the above decomposition of the identity is reasonable.
Then
$M_j(\xi,\eta)
=\sum_{i=1}^4M_{j,\De_i}
(\xi,\eta),$
where
$$
M_{j,\De_i}
(\xi,\eta)=\sum_{(m,n)\in \De_i}\sum_{k\in\Z}\widehat{\Phi}(\frac{\xi}{2^{jl+m}})
\widehat{\Phi}(\frac{\eta}{2^{k}})
\widehat{\Phi}(\frac{u(x)}{2^{j-k+n}})M_j(\xi,\eta),
$$
 and it suffices to show the following theorem.
\begin{thm}\label{t2}
For $1\le \sigma\le 4$, define
\begin{equation}\label{aim}
\h_{\De_\sigma}f(x,y)=\int_{\xi,\eta}\widehat{f}(\xi,\eta)
e(\xi x+\eta y) \sum_{j\in S_l} M_{j,\De_\sigma}(\xi,\eta)d\xi d\eta,
\end{equation}
then we have
$\|\h_{\De_\sigma}f\|_p\lesssim_N \|f\|_p$
holds for $1<p<\infty$.
\end{thm}
We end this section with  an  explanation of (\ref{dde}).
The low frequency component $\h_{\De_1}f$ is not 
oscillatory,  and can be seen as a more well-behaved integral, see section \ref{ll}.  We  treat the mixed frequency components $\h_{\De_2}f$ and $\h_{\De_3}f$ by integration by parts and square function estimate because of   the rapid decay stemming from  $|m-n|\ge 100l$, see section \ref{lh}.   The difference between  $\h_{\De_2}f$ and $\h_{\De_3}f$ is that we shall use Taylor's expansion to $\h_{\De_3}f$ before applying integration by parts.
At last, the high frequency component $\h_{\De_4}f$ needs a more intricate analysis,
which depends on a crucial estimate given by  Lemma \ref{ll2}, see section \ref{hh}.

%\begin{rem}\label{r11}
 %Observe that the feature of $\De_3$ is similar to the %combination of the features of $\De_1$ and $\De_2$, which %implies that
 % the $L^p$ estimate of  $\h_{\De_3}f$ can be achieved  by %combining the approaches yielding $\|\h_{\De_1}f\|_p$ and $\|%\h_{\De_2}f\|_p$.
%\end{rem}
\section{Auxiliary consequences}
\label{AC}
The following auxiliary lemmas are important  in proving the uniform  estimate of $\h_{\De_4}$.
Let $N$ be a positive integer, $C_0$, $C_1$ and $C_2$ be three uniform constants. Let  $\mathcal{F}(t)$ be a polynomial of degree  not more than $N$ which is  supported  in  $S:=\{t\in\R:\ C_0^{-1}\le |t |\le C_0\}$ and   satisfies   $C_1^{-1}\le |\mathcal{F}'(t)|\le C_1$ and $|\mathcal{F}^{''}(t)|\le C_2$ in $S$. Denote the inverse function of $\mathcal{F}$ by $F$. Obviously,
\begin{equation}\label{perty}
C_0^{-1}\le |F|\le C_0,\ F'(\tau)=\frac{1}{\mathcal{F}'(F(\tau))},\ |F'|\ge C_1^{-1}.
\end{equation}
\begin{lemma}\label{l31}
Let $\mathfrak{S}=\{s_i\}_{i=1}^{2^N}\subset [-\mathfrak{C}_0,\mathfrak{C}_0]$ with certain uniform $\mathfrak{C}_0>0$, and $\mathfrak{A}=\{\alpha_j\}_{j=1}^N\subset\R$ be two strictly increasing sets. Set $d(\mathfrak{S})=\min_{1\le i<j\le 2^N}|s_j-s_i|$,
$D_j(\mathfrak{A})=\prod_{1=i\neq j}^N|\alpha_j-\alpha_i|$ if $N\ge2$ and $D_j(\mathfrak{A})=1$ if $N=1$.
 %Let $\{b_j(s)\}_{j=1}^N$  be a series of functions %satisfying
 %\begin{equation}\label{lbb}
%\inf_{1\le j\le N}\inf_{s\in %[-\mathfrak{C}_0,\mathfrak{C}_0]}|b_j(s)|\ge C_1^{-1}.
%\end{equation}
 There are a constant $A_0>0$ and $\{a_j\}_{j=1}^N\subset \R$
such that
\begin{equation}\label{coo}
|\sum_{j=1}^N a_jF(s)^{\alpha_j}|\le A_0 \ \   {\rm\ for\ any}\ \ s\in \mathfrak{S},
\end{equation}
 then there exists a positive constant $\tilde{C}$ depending only on $\{\alpha_i\}_{i=1}^N$, $C_1$ and $C_0$ such that 
\begin{equation}\label{aa1}
|a_j|\le \frac{\tilde{C}^{N\big((\alpha_N-\alpha_1)+\min_{1\le j\le N}|\alpha_j|+2\big)}A_0}{\big(d(\mathfrak{S})\big)^{N-1} D_j(\mathfrak{A})}
\end{equation}
holds for  $1\le j\le N$.
In particular,  $\tilde{C}=  2C_1C_0^{\max_{1\le i\le N}|\alpha_i|+1}$ is enough.
\end{lemma}
\begin{rem}\label{rp}
The function $F$ in this lemma can be relaxed to any smooth function satisfying the first and the third conditions in (\ref{perty}). In addition, we require  only  (\ref{aa1}) with  ``$\min_{1\le j\le N}$" replaced by  ``$\max_{1\le j\le N}$" in the following context.
\end{rem}
\begin{proof}
We shall prove (\ref{aa1}) with $\tilde{C}\ge  2C_1C_0^{\max_{1\le i\le N}|\alpha_i|+1}$  by induction over the values of $N$. Obviously, $N=1$ is trivial  since
$|a_1|\le \frac{A_0}{|F(s)|^{\alpha_1}}\le   C_0^{|\alpha_1|} A_0\le \tilde{C} A_0$
holds for any $s\in \mathfrak{S}$.
We now assume that (\ref{aa1}) holds for  $N=k$,
 it suffices to show (\ref{aa1})  for $N=k+1$. Multiplying both sides of 
 \begin{equation}\label{k0}
|\sum_{j=1}^{k+1}a_jF(s)^{\alpha_j}|\le A_0\ \   {\rm\ for\ any}\ \ s\in \mathfrak{S},
\end{equation}
 %for any $s\in \mathfrak{S}$,
  by $|F(s)|^{-\alpha_1}$, we  obtain  
\begin{equation}\label{k11}
 |a_1+\sum_{j=2}^{k+1}a_jF(s)^{\alpha_j-\alpha_1}|\le A_0|F(s)|^{-\alpha_1}\le  C_0^{|\alpha_1|}A_0
 \end{equation}
holds  for any $s\in \mathfrak{S}$.
 %To cancel the resulting constant $a_1$, we plan to use
 Applying the mean value theorem to
  every interval  $(s_{2l-1},s_{2l})$ with $l=1,2,3\cdot\cdot\cdot,2^{N-1}$, we derive a collection of intermediate points $\{\tilde{s}_l\}_{l=1}^{2^N-1}$, which we denote by $\mathfrak{S}_\star$. Thanks to this process, we have obtained a new object like the summation on the left side of (\ref{k11}) without   the constant term $a_1$. More precisely,
 for all $s\in \mathfrak{S}_\star$,  we have
$$\frac{2C_1C_0^{|\alpha_1|+1}A_0}{d(\mathfrak{S})}\ge C_1C_0 |\sum_{j=2}^{k+1}(\alpha_j-\alpha_1)a_jF(s)^{\alpha_j-\alpha_1-1}
F'(s)|\ge |\sum_{j=2}^{k+1}(\alpha_j-\alpha_1)a_j
F(s)^{\alpha_j-\alpha_1}|,$$
which, with  the assumption (\ref{aa1}) for  $N=k$, leads  to that 
$$|a_j|(\alpha_j-\alpha_1)
\le\ \frac{\tilde{C}^{k((\alpha_{k+1}-\alpha_2)+\min_{2\le j\le k+1}|\alpha_j-\alpha_1|+2)}}{(d(\mathfrak{S_\star}))^{k-1} D_j(\mathfrak{A_\star})}
\frac{2C_1C_0^{|\alpha_1|+1}A_0}{d(\mathfrak{S})}$$
holds for all $2\le j\le k+1$,
where $\mathfrak{A_\star}=\{\alpha_j-\alpha_1\}_{j=2}^{k+1}$.
Note that  $\{(\alpha_j-\alpha_1)a_j\}_{j=2}^{k+1}$ are the new coefficients. Utilizing 
$d(\mathfrak{S})\le d(\mathfrak{S}_\star)$ which is deduced from the choice of $\{(s_{2l-1},s_{2l})\}_{l=1}^{2^{N-1}}$, $(\alpha_j-\alpha_1)D_j(\mathfrak{A_\star})=D_j(\mathfrak{A}),
$
and the choice of $\tilde{C}$, 
%$\ge C_1C_0^{\max_{1\le i\le N}\{|\alpha_i|\}+1},$$
 we obtain (\ref{aa1}) with $N$ replaced by $k+1$ holds for $2\le j\le k+1$.
Thus it remains to show the desired estimate of $a_1$. As a matter of fact,  
it can be treated by a similar way.  Multiplying both sides of   (\ref{k0}) by $|F(s)|^{-\alpha_{k+1}}$ and following the arguments below (\ref{k0}),   we  can also  obtain   the estimate of $a_1$, which completes the proof of Lemma \ref{l31}.
\end{proof}
Next, we introduce a lemma giving an effective control of certain sparse set called ``bad" set in section \ref{s1}. More importantly, it is essential in proving  the uniform  estimate of $\h_{\De_4}f$. 
Define
$$B(x,t):=\sum_{k=1}^Na_k(x)(F(t))^{\alpha_k},$$
where $F$ is defined as the previous statement,  $\alpha_1<\cdot\cdot\cdot<\alpha_N$, $\{\alpha_k\}_{k=1}^N \subset \R$, $x\in X\subseteq [-\mathfrak{C}_1,\mathfrak{C}_1]$ for certain absolute constant $\mathfrak{C}_1>0$  and $\{a_k(x)\}_{k=1}^N$ is a series of measurable functions satisfying $|a_k(x)|\lesssim 1$ for all $1\le k\le N$. Denote $\vec{\alpha}:=\{\alpha_1,\cdot\cdot\cdot,\alpha_N\}$.
\begin{lemma}\label{p1}
Let $m$  and $\epsilon$ be two positive   constants satisfying $\epsilon m\ge 1$.   Denote
$$D_m=\{w\in 2^{-\frac{m}{2}}\Z:\ C_5^{-1}\le |w|\le C_5\},
\ \mathfrak{D}_m=\{l\in 2^{-\frac{m}{2}}\Z:\ |l|\le C_4 \}$$
where $C_4,C_5\ge 1$ are two  absolute constants,
and
$$
\mathfrak{G}_{B,\epsilon}:=
\{(l,w)\in\mathfrak{D}_m\times D_m:\ |A_{B,\epsilon}(l,w)|\le 2^{-2\epsilon m}\},\ \ \mathfrak{H}_{B,\epsilon}:=(\mathfrak{D}_m\times D_m) \setminus \mathfrak{G}_{B,\epsilon},$$
where
$$A_{B,\epsilon}(l,w):=\{x\in X:\ |B(x,x+l)-w^{-1}|
\le 2^{-(1/2-2\epsilon)m}\}.$$
Denote
$$\mathfrak{H}_{B,\epsilon}^1:=\{l\in \mathfrak{D}_m:\ \exists\ w\in D_m,\  s.t.\  (l,w)\in \mathfrak{H}_{B,\epsilon}\}.$$
If
\begin{equation}\label{lower}
\inf_{x\in X}|\frac{\p}{\p t} B(x,t)|\gtrsim1,
\ m\ge \aleph_1,\ \epsilon\le\aleph_2^{-1}
\end{equation}
 for certain large enough constants $\aleph_1=\aleph_1(N,\vec{\alpha})$
 and $\aleph_2=\aleph_2(N)$,
 then there exists a positive constant $\mu=\mu(N)\in (0,1/2)$  such that
 \begin{equation}\label{sign}
\sharp\mathfrak{H}_{B,\epsilon}^1\lesssim \ 2^{(\frac{1}{2}-\mu)m}.
\end{equation}
\end{lemma}
\begin{rem}\label{rr1}
$``\aleph_1(N,\vec{\alpha})"$ in (\ref{lower}) can be  replaced by $``\aleph_1(N,\alpha_N)"$   when
$\alpha_i\in \N$ for $1\le i\le N$. In particular,
it can be replaced by
$``\aleph_1(N)"$ in our following proof.  Besides, 
 The restriction   $F=\mathcal{F}^{-1}$ is crucial  in this lemma,  we do not know that   whether this restriction  can be relaxed.
\end{rem}
\begin{proof}
To prove the desired estimate, we will use reduction ad absurdum. We assume that
\begin{equation}\label{As1}
\sharp\mathfrak{H}_{B,\epsilon}^1\ge  2^{(\frac{1}{2}-\mu)m},
\end{equation}
where $\frac{1}{2}-\mu\ge 2^{8N+2}\epsilon$. Our strategy is to prove that $\inf_{x\in Y}|\frac{\p}{\p t} B(x,t)|$ is smaller than any given positive constant for certain $Y\subset X$.
\vskip.1in
Denote $\lambda_{m}^{\epsilon,\mu}:=2^{2^{4N}2\epsilon m+\mu m}$.
We first construct   a sparse set $\mathfrak{H}_{B,\epsilon}^{1,1}\subset\mathfrak{H}_{B,\epsilon}^1$, which satisfies
\begin{equation}\label{abc1}
\sharp \mathfrak{H}_{B,\epsilon}^{1,1}
\thicksim\lambda_{m}^{\epsilon,0},\ \ \tilde{d}:=\inf_{\rho_1,\rho_2\in \mathfrak{H}_{B,\epsilon}^{1,1}}|\rho_1-\rho_2|\gtrsim (\lambda_{m}^{\epsilon,\mu})^{-1}.
\end{equation}
(Obviously, $\tilde{d}\lesssim1$). Indeed, set $\vartheta_i:=[i,i+1](\lambda_{m}^{\epsilon,\mu})^{-1}$, %denote
%$$M_0:=\inf\{M\in\Z:\ M>0,\ \ %\big(\bigcup_{i=-M}^{M}\vartheta_i\big) \supseteq %[-C_4,C_4]\},$$
%so $M_0\thicksim C_4\lambda_{m}^{\epsilon,\mu}$. %Furthermore, for any $-M_0\le i\le M_0$,
we have
$\sharp(\{\vartheta_i\cap 2^{-\frac{m}{2}}\Z\})\thicksim 2^\frac{m}{2}(\lambda_{m}^{\epsilon,\mu})^{-1},$
which, with  (\ref{As1}),
%$\sharp\mathfrak{H}_{B,\epsilon}^1\lesssim %2^\frac{m}{2} $, 
leads to that  there are at least $\thicksim\lambda_{m}^{\epsilon,0}$ intervals denoted by $\{\vartheta_{j_l}\}_{l=1}^{\mathfrak{M}_1}$ ($\mathfrak{M}_1\gtrsim \lambda_{m}^{\epsilon,0}$,  $j_1<j_2<\cdot\cdot\cdot<j_{\mathfrak{M}_1}$) such that $\vartheta_{j_l}\cap \mathfrak{H}_{B,\epsilon}^1\neq\varnothing$ for all $1\le l\le \mathfrak{M}_1$.
%Without loss of generality, we assume $z_1\le z_2$ for all %$(z_1,z_2)\in \vartheta_{j_1}\times \vartheta_{j_2}$ and %$j_1\le j_2$.
Denote $\mathfrak{S}_l:=\vartheta_{j_l}\cap \mathfrak{H}_{B,\epsilon}^{1}$,
choosing  any   point in  every  set  $\mathfrak{S}_{l}$ with odd $l$ yields  the desired set $\mathfrak{H}_{B,\epsilon}^{1,1}$.
\vskip.1in
Define
$|A_{B,\epsilon}(l,w_l)|=\max_{w:(l,w)\in \mathfrak{H}_{B,\epsilon}}|A_{B,\epsilon}(l,w)|,$
applying  Lemma \ref{la1} with   $I_l:=A_{B,\epsilon}(l,w_l)$,
$n:=2^{2N}$, $K:=2^{2\epsilon m}$ and $\tilde{N}:=2^{2^{4N}2\epsilon m}$, we gain that  
there exist $\mathfrak{H}_{B,\epsilon}^{1,1,1}\subset \mathfrak{H}_{B,\epsilon}^{1,1}$ with $\sharp \mathfrak{H}_{B,\epsilon}^{1,1,1}=n$ and $X_1:=\cap_{l\in \mathfrak{H}_{B,\epsilon}^{1,1,1}}I_l$ with $|X_1|\ge 2^{-1}K^{-n}$ such that for all $l\in  \mathfrak{H}_{B,\epsilon}^{1,1,1}$ and $x\in X_1$, we have
\begin{equation}\label{aa2}
|B(x,x+l)-w_l^{-1}|
\le 2^{-(1/2-2\epsilon)m}.
\end{equation}
 We next construct a sparse subset $X_1^1$ of $X_1$. 
 %In fact,
%the sparse property is crucial.
%When we use some functions taking value (say $x_1$) in %$X_1$ to approximate
Let  $E$ be  the collection of same length (= $2^{-\frac{m}{4}}$) intervals which partition   $[-\mathfrak{C}_1,\mathfrak{C}_1]$ and are mutually disjoint.
Note that the number of the intervals in $E$ is in $[\mathfrak{C}_1 2^{\frac{m}{4}+1},\mathfrak{C}_1 2^{\frac{m}{4}+1}+1]$.   Let $E_1$ be the collection of intervals $J\in E$  such that
$|X_1\cap J|\ge (2\mathfrak{C}_1)^{-1}2^{-4}K^{-n}|J|.$
Denote  $X_1^1:=\cup_{J\in E_1}(X_1\cap J)$,  we claim 
\begin{equation}\label{1.1}
|X_1'|\ge \frac{|X_1|}{2}\ge 2^{-2}K^{-n}.
\end{equation}
Actually, (\ref{1.1}) follows  since 
$$|X_1|\le\sum_{J\in E_1}|X_1\cap J|+\sum_{J\in E\setminus E_1}|X_1\cap J|
=|X_1'|+\sum_{J\in E\setminus E_1}|X_1\cap J|
$$
and
$$\sum_{J\in E\setminus E_1}|X_1\cap J|
\le\ (2\mathfrak{C}_1)^{-1} 2^{-4}K^{-n} 2^{-m/4} \sharp\{J:\ J\in E\setminus E_1\}\le\ 2^{-4}K^{-n}\le \frac{|X_1|}{2}.$$
We conclude  that for  $y\in X_1'$, there exist $y'$ and $J$ such that $y'\in J \in E_1$  and
$ |J|\ge|y-y'|\ge (2\mathfrak{C}_1)^{-1}2^{-5}K^{-n}|J|.$
Thus, for any $x\in X_1^1$, applying Lemma \ref{l31} to (\ref{aa2}), we deduce that 
%(we replace $``\min"$ in (\ref{aa1}) by $``\max"$ for some %convenience)
for all $1\le j\le N$,
\begin{equation}\label{aa3}
|a_j(x)|\le \frac{\tilde{C}^{N((\alpha_N-\alpha_1)+\max_{1\le j\le N}|\alpha_j|+2)}(1+C_5)}{\tilde{d}^{N-1} \inf_{j}\prod_{1=i\neq j}^N|\alpha_j-\alpha_i|}:=\frac{C_{\vec{ \alpha},N}}{\tilde{d}^{N-1}},
\end{equation}
in which $C_{\vec{\alpha},N}$ (which may vary at each appearance below),
and there exists $x'\in X_1'$ satisfying
\begin{equation}\label{sparse} (2\mathfrak{C}_1)^{-1}2^{-5}K^{-n}|J|\le|x-x'|\le |J|
\end{equation}
such that  for all $l\in \mathfrak{H}_{B,\epsilon}^{1,1,1}$,
$|B(x,x+l)-B(x',x'+l)|
\le 2^{1-(1/2-2\epsilon)m},$
namely,
$|\sum_{k=1}^Na_k(x)(F(x+l))^{\alpha_k}-\sum_{k=1}^Na_k(x')
(F(x'+l))^{\alpha_k}|\le  2^{1-(1/2-2\epsilon)m}.$
Now,
taking advantage  of Taylor's  expansion
$$(F(x'+l))^{\alpha_k}-(F(x+l))^{\alpha_k}
=\alpha_k (x'-x)F'(x+l)F(x+l)^{\alpha_k-1}+O(|x-x'|^2),$$
where $O(|x-x'|^2)$ means a term $\lesssim_{\alpha_k,C_1,C_2}|x-x'|^2$,
and using (\ref{aa3}),
we can obtain
\begin{equation}\label{set1}
\begin{aligned}
&\ \ |\sum_{k=1}^N[a_k(x)-a_k(x')](F(x+l))^{\alpha_k}-
\sum_{k=1}^N\alpha_k a_k(x') (x-x')F'(x+l)F(x+l)^{\alpha_k-1}|\\
\lesssim&\  2^{-(1/2-2\epsilon)m}+\frac{C_{\vec{ \alpha},N}}{\tilde{d}^{N-1}}O(|x-x'|^2).
\end{aligned}
\end{equation}
Thanks to (\ref{aa3}) and the upper bound of $|x-x'|$, we deduce from (\ref{set1}) that
\begin{equation}\label{set2}
\begin{aligned}
|\sum_{k=1}^N[a_k(x)-a_k(x')](F(x+l))^{\alpha_k}|
\lesssim&\  2^{-(1/2-2\epsilon)m}+\frac{C_{\vec{ \alpha},N}}{\tilde{d}^{N-1}}O(|x-x'|)+\frac{C_{\vec{ \alpha},N}}{\tilde{d}^{N-1}}O(|x-x'|^2)\\
\lesssim&\ 2^{-(1/2-2\epsilon)m}+\frac{C_{\vec{ \alpha},N}}{\tilde{d}^{N-1}}O(|x-x'|).
\end{aligned}
\end{equation}
Applying Lemma \ref{l31} to (\ref{set2}), we have  for $1\le k\le N$,
\begin{equation}\label{az1}
|a_k(x)-a_k(x')|\lesssim\ \frac{C_{\vec{\alpha},N}}{\tilde{d}^{N-1}} \big(2^{-(1/2-2\epsilon)m}+\frac{C_{\vec{ \alpha},N}}{\tilde{d}^{N-1}}O(|x-x'|)\big).
\end{equation}
Note that we can not directly use Lemma \ref{l31} to the object on the left side
of (\ref{set1}) because the coefficient $F'(x+l)$ depends on $l$. However, thanks to (\ref{perty}),  $F'(x+l)=\frac{1}{\mathcal{F}'(F(x+l))}$ and the fact that  $\mathcal{F}$ is a polynomial of degree not more than $N$, we can  get around  this  barrier. Multiplying (\ref{set1}) by $\mathcal{F}'(F(x+l))$,  using $|\mathcal{F}'(F(x+l))|\le C_1$ and the fact that  $\mathcal{F}'$ is a  polynomial of degree not more  than $N-1$, we  obtain
\begin{equation}\label{set11}
\begin{aligned}
&\ \ |\sum_{k=1}^N[a_k(x)-a_k(x')](F(x+l))^{\alpha_k}\mathcal{F}'(F(x+l))-
\sum_{k=1}^N\alpha_k a_k(x') (x-x')F(x+l)^{\alpha_k-1}|\\
\lesssim&\  2^{-(1/2-2\epsilon)m}+\frac{C_{\vec{ \alpha},N}}{\tilde{d}^{N-1}}O(|x-x'|^2).
\end{aligned}
\end{equation}
Observe that the lowest power on the left side of (\ref{set11}) is $\alpha_1-1$, it follows by Lemma \ref{l31} and (\ref{sparse}) that
%the coefficient of the monomial $\alpha_1 a_1(x') %(x-x')F(x+l)^{\alpha_1-1}$ is bounded by
\begin{equation}\label{e10}
\begin{aligned}
|a_1(x')|
\lesssim&\ \frac{C_{\vec{\alpha},N}}{|x-x'|\tilde{d}^{N^2-1}}( 2^{-(1/2-2\epsilon)m}+\frac{1}{\tilde{d}^{N-1}}O(|x-x'|^2))\\
\lesssim&\  \frac{C_{\vec{\alpha},N}}{\tilde{d}^{N^2+N}} 2^{-(1/4-2\epsilon)m}2^{2^{2N}2\epsilon m},
\end{aligned}
\end{equation}
%To obtain the estimate of $a_1(x)$,  it seems necessary to discuss %$\alpha_2-1=\alpha_1$ or $\alpha_2-1\neq\alpha_1$ if we directly use Lemma \ref{l31}. %To avoid this, we deal with a new objective, which is obtained from
% (\ref{set1}), namely,
%\begin{equation}\label{set2}
%\begin{aligned}
%|\sum_{k=1}^N[a_k(x)-a_k(x')](F(x+l))^{\alpha_k}|
%\lesssim&\  2^{-(1/2-2\epsilon)m}+\frac{C_{\vec{ %\alpha},N}}{\tilde{d}^{N-1}}O(|x-x'|)+\frac{C_{\vec{ %\alpha},N}}{\tilde{d}^{N-1}}O(|x-x'|^2)\\
%\lesssim&\ 2^{-(1/2-2\epsilon)m}+\frac{C_{\vec{ \alpha},N}}{\tilde{d}^{N-1}}O(|x-x'|).
%\end{aligned}
%\end{equation}
%Applying Lemma \ref{l31} to (\ref{set2}) yields
%$$|a_1(x)-a_1(x')|\lesssim\ \frac{1}{\tilde{d}^N} %\big(2^{-(1/2-2\epsilon)m}+\frac{C_{\vec{ %\alpha},N}}{\tilde{d}^{N-1}}O(|x-x'|)\big),$$
which, with (\ref{az1}), implies
$
|a_1(x)|+|a_1(x')|
\lesssim\  \frac{C_{\vec{\alpha},N}}{\tilde{d}^{N^2+N}} 2^{-(1/4-2\epsilon)m}2^{2^{2N}2\epsilon m}.
$
So
$$|\sum_{k=2}^Na_k(x)(F(x+l))^{\alpha_k}-\sum_{k=2}^Na_k(x')
(F(x'+l))^{\alpha_k}|\lesssim \frac{C_{\vec{\alpha},N}}{\tilde{d}^{N^2+N}} 2^{-(1/4-2\epsilon)m}2^{2^{2N}2\epsilon m}.
$$
Repeating the above process $N-1$ times, we get
for any $1\le k\le N$, $x\in X_1^N$ with $X_1^N\subset X_1^{N-1}\subset\cdot\cdot\cdot\subset X_1^1\subset X_1$
and $|X_1^N|\ge \frac{|X_1|}{2^{N+1}}$ that
$$
\begin{aligned}
|a_k(x)|+|a_k(x')|
\lesssim&\  \frac{C_{\vec{\alpha},N}}{\tilde{d}^{2N^2k}} 2^{-(\frac{1}{2^{k+1}}-2\epsilon)m}2^{2^{2N}2k\epsilon m}.
\end{aligned}
$$
Utilizing  (\ref{abc1}), we deduce
$$|a_k(x)|\lesssim C_{\vec{\alpha},N}(2^{2^{4N}2\epsilon m+\mu m})^{2N^2k} 2^{-(\frac{1}{2^{k+1}}-2\epsilon)m}2^{2^{2N}2k\epsilon m}.$$
Thanks to the conditions on $m$ and $\epsilon$ in (\ref{lower}), the right side is
$\lesssim\ C_{\vec{\alpha},N}2^{-\varepsilon_1 m}$
for certain positive constant $\varepsilon_1=\varepsilon_1(N)$. This gives
$$1\stackrel{(\ref{lower})}{\lesssim}\inf_{x\in X}|\frac{\p}{\p t} B(x,t)|\le
\inf_{x\in X_1^N}|\frac{\p}{\p t} B(x,t)|\lesssim\ C_{\vec{\alpha},N}2^{-\varepsilon_1 m}$$
  which yields a contradiction.  Therefore,
  (\ref{As1}) does not hold, which
   completes the proof.
\end{proof}

\vskip.3in
\section{Proof of Theorem \ref{t2} for $\sigma=1$: $L^p$ estimate of $\h_{\De_1}f$}
\label{ll}
In this section, we prove Theorem \ref{t2} for $\sigma=1$.
%the desired $L^p$ estimate of $\h_{\De_1}f$, which is defined %by (\ref{aim}) for $i=1$.
Recall
$$
\De_1=\{(m,n)\in\Z^2:\ m\le 0,\ n\le 0\},\
M_j(\xi,\eta)=\int e(\xi P(2^{-j}t)+\eta u(x)2^{-j}t)\rho(t)dt
$$
and
$$M_{j,\De_1}
(\xi,\eta)=\sum_{(m,n)\in \De_1}\sum_{k\in\Z}\widehat{\Phi}(\frac{\xi}{2^{jl+m}})
\widehat{\Phi}(\frac{\eta}{2^{k}})
\widehat{\Phi}(\frac{u(x)}{2^{j-k+n}})M_j(\xi,\eta).$$
Denote
\begin{equation}\label{dff1}
\wp:=\{(n_1,n_2)\in \Z^2:\ n_1\ge 0,\ n_2\ge 0,\ n_1+n_2>0\},\ \ \widehat{\Phi_{n_3}}(\xi):=\xi^{n_3}\widehat{\Phi}(\xi),\ n_3\in\Z.
\end{equation}
Since $(m,n)\in \De_1$,
 $j\in S_l$, the support of $\widehat{\Phi}(\cdot)$, via  Taylor's expansions
$$e(\xi P(2^{-j}t))=\sum_{q\ge 0}\frac{i^q}{q!}(\xi P(2^{-j}t))^q=2^{mq}\sum_{q\ge 0}\frac{i^q}{q!}(\frac{\xi}{2^{m+jl}})^q(2^{jl} P(2^{-j}t))^q$$
and
$$e(\eta u(x)2^{-j}t)=\sum_{\upsilon\ge 0}\frac{i^\upsilon}{\upsilon!}(\eta u(x)2^{-j}t)^\upsilon=2^{n\upsilon}\sum_{\upsilon\ge 0}\frac{i^\upsilon}{\upsilon!}(\frac{\eta}{2^k})^\upsilon
(\frac{u(x)}{2^{j-k+n}})^\upsilon t^\upsilon,$$
we have  
$$
\begin{aligned}
M_{j,\De_1}
(\xi,\eta)=&\  \sum_{(m,n)\in \De_1}\sum_{k\in\Z}\sum_{(q,\upsilon)\in \wp}
\frac{i^{q+\upsilon}}{q!\upsilon!} 2^{mq+n\upsilon}
\widehat{\Phi_q}(\frac{\xi}{2^{jl+m}})
\widehat{\Phi_\upsilon}(\frac{u(x)}{2^{j-k+n}})
\widehat{\Phi_\upsilon}(\frac{\eta}{2^{k}}) \gamma_{j,q,\upsilon},
\end{aligned}
$$
where
$\gamma_{j,q,\upsilon}:=\int_\R (2^{jl}P(2^{-j}t))^q t^\upsilon \rho(t)dt.$
It is not hard to see the  tricky cases are $q=0$ and  $\upsilon=0$, which make us
 split the summation of $(q,\upsilon)$ into three parts:
$\sum_{q=0,\upsilon\ge 1}+\sum_{\upsilon=0,q\ge 1}+\sum_{\upsilon\ge 1,q\ge 1}.$
 In the following, we  only give the detailed proof of   the former two terms since it is easier to dominate the third term.
 \vskip.1in
  Denote
$$
M_{j,\De_1,1}(\xi,\eta):=\sum_{(m,n)\in \De_1}\sum_{k\in\Z}\sum_{\upsilon\ge1}
\frac{i^{\upsilon}}{\upsilon!} 2^{n\upsilon}
\widehat{\Phi}(\frac{\xi}{2^{jl+m}})
\widehat{\Phi_\upsilon}(\frac{u(x)}{2^{j-k+n}})
\widehat{\Phi_\upsilon}(\frac{\eta}{2^{k}}) \gamma_{j,0,\upsilon}
$$
and
$$M_{j,\De_1,2}(\xi,\eta):=\sum_{(m,n)\in \De_1}\sum_{k\in\Z}\sum_{q\ge 1}
\frac{i^{q}}{q!} 2^{mq}
\widehat{\Phi_q}(\frac{\xi}{2^{jl+m}})
\widehat{\Phi}(\frac{u(x)}{2^{j-k+n}})
\widehat{\Phi}(\frac{\eta}{2^{k}}) \gamma_{j,q,0},$$
%to get the desired $L^p$ estimate of $\h_{\De_1}f$,
it suffices to show that for $\sigma=1,2$,
\begin{equation}\label{A1}
\|\h_{\De_1,\sigma}f\|_p\lesssim_N \|f\|_p,
\end{equation}
where
$$
\h_{\De_1,\sigma}f(x,y):=\int_{\xi,\eta}\widehat{f}(\xi,\eta)
e(\xi x+\eta y) \sum_{j\in S_l} M_{j,\De_1,\sigma}(\xi,\eta)d\xi d\eta.
$$
\subsection{The estimate of $\h_{\De_1,1}f(x,y)$ }
We will prove (\ref{A1}) for $\sigma=1$.
 By dual arguments, it is enough to show that for all $g\in L^{p'}(\R^2)$,
 \begin{equation}\label{D2}
|\langle \h_{\De_1,1}f,g \rangle|
\lesssim_N\ \|f\|_p\|g\|_{p'}.
\end{equation}
Denote
\begin{equation}\label{df2}
\widehat{\hbar}(\xi):=\sum_{m\le 0}\widehat{\Phi}(\frac{\xi}{2^m}),\  \hbar_s(x):=2^s\hbar(2^sx), \ \Phi_{\upsilon,k}(x):=2^k\Phi_{\upsilon}(2^kx),\  \Psi_k(y):=(\sum_{|k'-k|\le1}\widehat{\Phi}(2^{-k'}\eta)
)^{\vee}(-y),
\end{equation}
where $|\Phi_{\upsilon}(x)|\lesssim\ 2^{2\upsilon}(1+|x|^2)^{-1}$,
it follows by Fourier inverse transform, $|\gamma_{j,0,\upsilon}|\lesssim 2^\upsilon$ and  H\"{o}lder's inequality that
\begin{equation}\label{123}
\begin{aligned}
{\rm LHS\  of\  (\ref{D2})}\lesssim&\ \sum_{n\le 0}\sum_{\upsilon\ge 1}
\frac{2^{n\upsilon}}{\upsilon!}
|\langle\sum_{j\in S_l}\sum_{k\in\Z}
\widehat{\Phi_\upsilon}(\frac{u(x)}{2^{j-k+n}}) \gamma_{j,0,\upsilon}
\hbar_{jl}*_x \Phi_{\upsilon,k}*_y f,\
\Psi_k*_yg
\rangle|\\
\lesssim&\ \sum_{n\le 0}\sum_{\upsilon\ge 1}
\frac{2^{(n+1)\upsilon}}{\upsilon!}
\Big\|\big(\sum_{j\in S_l}\sum_{k\in\Z}|\widehat{\Phi_\upsilon}
(\frac{u(x)}{2^{j-k+n}})| |\hbar_{jl}*_x \Phi_{\upsilon,k}*_y f|^2\big)^\frac{1}{2}\Big\|_p\\
&\ \times \Big\|\big(\sum_{j\in S_l}\sum_{k\in\Z}|\widehat{\Phi_\upsilon}
(\frac{u(x)}{2^{j-k+n}})| | \Psi_k*_yg|^2\big)^\frac{1}{2}\Big\|_{p'}.
\end{aligned}
\end{equation}
Because of $\sum_{j\in S_l}|\widehat{\Phi}_\upsilon
(\frac{u(x)}{2^{j-k+n}})|\lesssim 2^\upsilon$,  by Littlewood-Paley theorem, the norm $\|\cdot\|_{p'}$ on the right side of (\ref{123}) is $\lesssim 2^{\upsilon/2}\|g\|_{p'}$.
In what follows, $M^{(1)}$ and $M^{(2)}$ denote the Hardy-Littlewood Maximal operators applied in the first  variable and the second variable, respectively.
Note that  $\hbar_{jl}*_x f\lesssim M^{(1)}f$.  Then  we bound  the norm $\|\cdot\|_p$ on the right side of (\ref{123})   by
$$
\begin{aligned}
\Big\|\big(\sum_{j\in S_l}\sum_{k\in\Z}|\widehat{\Phi}_\upsilon
(\frac{u(x)}{2^{j-k+n}})| |M^{(1)}( \Phi_{\upsilon,k}*_y f)|^2\big)^\frac{1}{2}\Big\|_p
\lesssim&\ 2^{\upsilon/2}\Big\|\big(\sum_{k\in\Z} |M^{(1)}( \Phi_{\upsilon,k}*_y f)|^2\big)^\frac{1}{2}\Big\|_p\\
\lesssim&\ 2^{\upsilon/2}\Big\|\big(\sum_{k\in\Z} | \Phi_{\upsilon,k}*_y f|^2\big)^\frac{1}{2}\Big\|_p\\
\lesssim&\ 2^{5\upsilon/2}\Big\|\big(\sum_{k\in\Z} | M^{(2)} (\tilde{\Psi}_k*_yf)|^2\big)^\frac{1}{2}\Big\|_p\\
\lesssim&\ 2^{5\upsilon/2}\|f\|_p,
\end{aligned}
$$
where $\tilde{\Psi}_k(y):=\Psi_k(-y)$,
 Fefferman-Stein's inequality and Littlewood-Paley theorem are applied.   Inserting these estimates into (\ref{123}) leads to that 
 ${\rm LHS\  of\  (\ref{D2})}\lesssim \|f\|_p\|g\|_{p'}\sum_{n\le 0}\sum_{\upsilon\ge 1}
\frac{2^{(n+1)\upsilon}8^\upsilon}{\upsilon!}.$
 Therefore, the proof of (\ref{D2}) is completed. 
 %since the double summation is $\lesssim1$.
\subsection{The estimate of $\h_{\De_1,2}f(x,y)$ }
We will prove (\ref{A1}) for $\sigma=2$.
Using $\sum_{n\le 0}\widehat{\Phi}(\frac{u(x)}{2^{j-l+n}})
=\widehat{\hbar}(\frac{u(x)}{2^{j-l}}),$
we have
$$M_{j,\De_1,2}(\xi,\eta)=\sum_{m\le0}
\sum_{k\in\Z}\sum_{q\ge 1}
\frac{i^{q}}{q!} 2^{mq}
\widehat{\Phi}_q(\frac{\xi}{2^{jl+m}})
\widehat{\hbar}(\frac{u(x)}{2^{j-k}})
\widehat{\Phi}(\frac{\eta}{2^{k}}) \gamma_{j,q,0}.$$
We remark that at this point it is different from the previous case since the summation of $j$ can not be absorbed by $\widehat{\hbar}(\frac{u(x)}{2^{j-k}})$ any more.
Our strategy is to make full use of   $\widehat{\Phi}_q(\frac{\xi}{2^{jl+m}})$.
\vskip.1in
Analogously,
  dual arguments gives that  it suffices to prove for all $g\in L^{p'}$,
 \begin{equation}\label{D3}
|\langle \h_{\De_1,2}f,g \rangle|
\lesssim\ \|f\|_p\|g\|_{p'}.
\end{equation}
Applying Fourier inverse transform,  H\"{o}lder's inequality and Littlewood-Paley theorem, we obtain
$$
\begin{aligned}
{\rm LHS\  of\ } (\ref{D3})\lesssim&\ \sum_{m\le 0}\sum_{q\ge 1}\frac{2^{mq}}{q!}
|\langle\sum_{j\in S_l}\sum_{k\in\Z}
\widehat{\hbar}(\frac{u(x)}{2^{j-k}}) \gamma_{j,q,0}
\Phi_{q,jl+m}*_x \Phi_{k}*_y f,\
\Psi_{k}*_yg
\rangle|\\
\lesssim&\ \Big\|\big(\sum_{k\in\Z}|\Psi_{k}*_yg|^2
\big)^\frac{1}{2}\Big\|_{p'}\sum_{m\le 0}\sum_{q\ge 1}\frac{2^{mq}}{q!}
\beth_{m,q}\lesssim\ \|g\|_{p'}\sum_{m\le 0}\sum_{q\ge 1}\frac{2^{mq}}{q!} \beth_{m,q},
\end{aligned}
$$
where
$$\beth_{m,q}:=\Big\|\big(\sum_{k\in\Z}|\sum_{j\in S_l}\widehat{\hbar}(\frac{u(x)}{2^{j-k}}) \gamma_{j,q,0}
\Phi_{q,jl+m}*_x \Phi_{k}*_y f|^2\big)^\frac{1}{2}\Big\|_{p}.$$
So it is enough to prove
$
\beth_{m,q}\lesssim\ 100^q\|f\|_p.
$
Density arguments yields that it suffices to prove for any $M_1>0$, measurable functions $ z(x)\in \Z$ and $Z(x)\in \Z$,
\begin{equation}\label{D5}
\tilde{\beth}_{m,q}\lesssim\ 100^q\|f\|_p,
\end{equation}
where
$$\tilde{\beth}_{m,q}:=\Big\|\big(\sum_{|k|\le M_1}|\sum_{j\in [z(x),Z(x)]}\chi_{S_l}(j)\widehat{\hbar}(\frac{u(x)}{2^{j-k}}) \gamma_{j,q,0}
\Phi_{q,jl+m}*_x \Phi_{k}*_y f|^2\big)^\frac{1}{2}\Big\|_{p}.$$
Note that if $z(x)=Z(x)$, (\ref{D5}) is immediately  achieved by applying Fefferman-Stein's inequality and Littlewood-Paley theorem. In what follows, we assume $z(x)\le Z(x)-1$. As the previous statement,
here we want to make full use of  $\Phi_{q,jl+m}$ to absorb the summation of $j$.
\vskip.1in
Recall  the process of  Abel summation, that is, denote $S_n=\sum_{\kappa\le n}J_\kappa$, we have
\begin{equation}\label{E1}
\begin{aligned}
 \sum_{\kappa=z(x)}^{Z(x)} H_\kappa J_\kappa
 =&\  \sum_{\kappa=z(x)}^{Z(x)} H_\kappa (S_\kappa-S_{\kappa-1})
 = \sum_{\kappa=z(x)}^{Z(x)} H_\kappa S_\kappa
 - \sum_{\kappa=z(x)-1}^{Z(x)-1} H_{\kappa+1} S_\kappa\\
 =&\ H_{Z(x)}S_{Z(x)}+ \sum_{\kappa=z(x)}^{Z(x)-1} (H_{\kappa}-H_{\kappa+1}) S_\kappa-H_{z(x)}S_{z(x)-1}.
\end{aligned}
\end{equation}
Applying (\ref{E1}) with
$H_j:=\widehat{\hbar}(\frac{u(x)}{2^{j-k}}),\ \ J_j:=\Phi_{q,jl+m}\chi_{S_l}(j)
\gamma_{j,q,0},$
the first and third term can be bounded by the same way as yielding the desired estimate for  the case $z(x)=Z(x)$. So we only pay attention to the second term.  That is, to prove (\ref{D5}), it suffices to show
 \begin{equation}\label{D6}
\tilde{\beth}_{m,q}^{(2)}\lesssim\ 100^q\|f\|_p,
\end{equation}
where
$$\tilde{\beth}_{m,q}^{(2)}:=
\Big\|\big(\sum_{|k|\le M_1}|\sum_{j\in [z(x),Z(x)-1]}\big(\widehat{\hbar}(\frac{u(x)}{2^{j-k}})
-\widehat{\hbar}(\frac{u(x)}{2^{j+1-k}})
\big)
\bar{\hbar}_{q,jl+m}*_x \Phi_{k}*_y f|^2\big)^\frac{1}{2}\Big\|_{p},$$
in which  $\bar{\hbar}_{q,jl+m}:=\sum_{j'\le j}\Phi_{q,j'l+m}\chi_{S_l}(j')\gamma_{j',q,0}$. 
%Notice that $\bar{\hbar}_{q,jl+m}*_x f\lesssim 2^{3q} M^{(1)}f$ since
%$\chi_{S_l}(j')\gamma_{j',q,0}\lesssim 2^q$.  
Hence,  to prove (\ref{D6}), it suffices to show
\begin{equation}\label{D7}
\Im(x):=\sup_{k}\sum_{j\in [z(x),Z(x)-1]}|\widehat{\hbar}(\frac{u(x)}{2^{j-k}})
-\widehat{\hbar}(\frac{u(x)}{2^{j+1-k}})
|\lesssim1.
\end{equation}
As a matter of fact, denote $\bar{\hbar}_{q,\infty}:=\sum_{j'\in\Z}\Phi_{q,j'l+m}\chi_{S_l}(j')\gamma_{j',q,0}$,
we have via (\ref{D7}), a variant of Cotlar's inequality,  Fefferman-Stein's inequality and Littlewood-Paley theorem that
$$
\begin{aligned}
 \tilde{\beth}_m^{(2)}\lesssim&\ 2^{3q}
 \Big\|\left(\sum_{|k|\le M_1}\left|
 \sup_{j\in [z(x),Z(x)]}\big|
\bar{\hbar}_{q,jl+m}*_x \Phi_{k}*_y f\big|\right|^2\right)^\frac{1}{2}\Big\|_{p}\\
\lesssim&\ 2^{3q}\Big\|\left(\sum_{|k|\le M_1}\left|
M^{(1)} (\Phi_{k}*_y f)\right|^2\right)^\frac{1}{2}\Big\|_{p}
+2^{3q}\Big\|\left(\sum_{|k|\le M_1}\left|
M^{(1)} (\bar{\hbar}_{q,\infty}*_x \Phi_{k}*_y f)\right|^2\right)^\frac{1}{2}\Big\|_{p}\\
\lesssim&\ 2^{3q}\|f\|_p.
\end{aligned}
$$
Thus it remains to prove (\ref{D7}).  By the fundamental theorem of calculus, we have
$$
\begin{aligned}
|\widehat{\hbar}(\frac{u(x)}{2^{j-k}})
-\widehat{\hbar}(\frac{u(x)}{2^{j+1-k}})|
=&\ |\int_0^1\frac{d}{ds}
\big(\widehat{\hbar}(\frac{u(x)}{2^{j-k}}s
+\frac{u(x)
}{2^{j+1-k}}(1-s)) \big)ds|\\\
=&\ |\frac{u(x)}{2^{j+1-k}}|\int_0^1|\widehat{\hbar}'
\big(\frac{u(x)}{2^{j+1-k}}(1+s)
\big)|ds
\lesssim\ \frac{|\frac{u(x)}{2^{j+1-k}}|
}{1+|\frac{u(x)}{2^{j+1-k}}|^2},
\end{aligned}
$$
Hence, (\ref{D7}) follows since
 $\sum_{j\in\Z}\frac{|\frac{u(x)}{2^{j+1-k}}|
}{1+|\frac{u(x)}{2^{j+1-k}}|^2}\lesssim1$. This completes the proof of (\ref{A1}).
\section{Proof of Theorem \ref{t2} for $\sigma=2,3$: $L^p$ estimates of $\h_{\De_2}f$ and $\h_{\De_3}f$}
\label{lh}
In this section, Theorem \ref{t2} for $\sigma=2,3$ will be proved. As the previous explanation, we only give the details for $\sigma=2$.
\subsection{The estimate of $\h_{\De_2}f$}
Without loss of generality, we assume that $m>n+100l$ and $n>0$. Remember
$$\phi_{j,\xi,\eta,x}(t)=\xi P(2^{-j}t)+\eta u(x)2^{-j}t,\ M_j(\xi,\eta)=\int_\R e(\phi_{j,\xi,\eta,x}(t)) \rho(t) dt$$
and the definition of $M_{j,\De_2}(\xi,\eta)$
$$
M_{j,\De_2}
(\xi,\eta)=\sum_{(m,n)\in \De_2}\sum_{k\in\Z}\widehat{\Phi}(\frac{\xi}{2^{jl+m}})
\widehat{\Phi}(\frac{\eta}{2^{k}})
\widehat{\Phi}(\frac{u(x)}{2^{j-k+n}})M_j(\xi,\eta).
$$
we  see that the phase function $\phi_{j,\xi,\eta,x}(t)$ does not have a critical point since
$m>n+100l$. In fact,  $|\phi_{j,\xi,\eta,x}'(t)|$ owns a large lower bound,  that is,
$$
\begin{aligned}
|\phi_{j,\xi,\eta,x}'(t)|=&\ |\xi 2^{-j}P'(2^{-j}t)+\eta u(x)2^{-j}|\\
\ge&\  2^m|2^{-jl-m}\xi||2^{j(l-1)}P'(2^{-j}t)|-
2^n|2^{-k}\eta||2^{-(j-k+n)}u(x)|\\
\gtrsim_l&\  2^m,
\end{aligned}
$$
which prompts us to employ integration by parts.
\vskip.1in
Motivated by the above analysis,  we have via utilizing  integration by parts
$$
\begin{aligned}
M_j(\xi,\eta)=&\ -i\int \frac{1}{\phi_{j,\xi,\eta,x}'(t)}\frac{d}{dt}(e^{i\phi_{j,\xi,\eta,x}(t)})
\rho(t)dt\\
=&\ i\int \frac{1}{\phi_{j,\xi,\eta,x}'(t)}e^{i\phi_{j,\xi,\eta,x}(t)}
\rho'(t)dt+i\int \frac{\phi_{j,\xi,\eta,x}''(t)
}{(\phi_{j,\xi,\eta,x}'(t))^2}e^{i\phi_{j,\xi,\eta,x}(t)}
\rho(t)dt.
\end{aligned}
$$
Since the above two terms can be homoplastically treated, we only focus on the first term. Taylor's expansion provides
$$
\begin{aligned}
\frac{1}{\phi_{j,\xi,\eta,x}'(t)}=&
\ \frac{1}{\xi 2^{-j}P'(2^{-j}t)}\frac{1}{1+\frac{u(x)\eta}{\xi P'(2^{-j}t)}}
=\ \frac{1}{\xi 2^{-j}P'(2^{-j}t)}\sum_{r=0}^\infty
(-\frac{u(x)\eta}{\xi P'(2^{-j}t)})^r\\
=&\ 2^{-m}\sum_{r=0}^\infty (-1)^r 2^{(n-m)r}(\frac{\xi}{2^{jl+m}})^{-r-1}
(\frac{u(x)}{2^{j-k+n}})^r (\frac{\eta}{2^k})^r
(2^{j(l-1)}P'(2^{-j}t))^{-r-1}.
\end{aligned}
$$
As the statement in section \ref{s1},  Taylor's expansion is used  to obtain the rapid decay.
Recall the definition of  $\widehat{\Phi_r}(\cdot)$ in (\ref{dff1}), and denote %$\widehat{\Phi_r}(\cdot)$=\widehat{\Phi}(\cdot)(\cdot)^{-r}$
%and
$\rho_{j,r}(t):=(2^{j(l-1)}P'(2^{-j}t))^{-r-1}\rho'(t)$,
$$
\begin{aligned}
M_{j,\De_2,m,n}(\xi,\eta):=&\ 2^{-m}\sum_{k\in\Z}\sum_{r=0}^\infty (-1)^r 2^{(n-m)r}
\widehat{\Phi_{-r-1}}(\frac{\xi}{2^{jl+m}})
\widehat{\Phi_r}(\frac{\eta}{2^{k}})
\widehat{\Phi_r}(\frac{u(x)}{2^{j-k+n}})\int_\R  e^{i\phi_{j,\xi,\eta,x}(t)}
\rho_{j,r}(t)dt,
\end{aligned}
$$
and
$M_{\De_2,m,n}(\xi,\eta):=\sum_{j\in S_l} M_{j,\De_2,m,n}(\xi,\eta),$
thus, in order to prove (\ref{aim}) for $\sigma=2$, it suffices to prove that
\begin{equation}\label{AA1}
\|\int_{\R^2}
e(\xi x+\eta y) \widehat{f}(\xi,\eta) M_{\De_2,m,n}(\xi,\eta) d\xi d\eta||_p\lesssim\ m^22^{-m}\|f\|_p.
\end{equation}
%where
%$$T_{m,n,}f(x,y):=\int_{\R^2}
%e(\xi x+\eta y) \widehat{f}(\xi,\eta) M_{\De_2,m,n}(\xi,\eta) %d\xi d\eta.$$
Denote
$$M_{r,m,n}(\xi,\eta):=\sum_{k\in\Z}\sum_{j\in S_l}
\widehat{\Phi_{-r-1}}(\frac{\xi}{2^{jl+m}})
\widehat{\Phi_r}(\frac{\eta}{2^{k}})
\widehat{\Phi_r}(\frac{u(x)}{2^{j-k+n}})\int_\R  e^{i\phi_{j,\xi,\eta,x}(t)}
\rho_{j,r}(t)dt$$
and
$$T_{r,m,n}f(x,y):=2^{-m}\int_{\R^2}
e(\xi x+\eta y) \widehat{f}(\xi,\eta) M_{r,m,n}(\xi,\eta) d\xi d\eta,$$
 To demonstrate  (\ref{AA1}), it suffices to show that there exists a positive constant $\mathfrak{H}\le 2^{m-n-1}$   such that for all $g\in L^{p'}(\R^2)$,
\begin{equation}\label{A41}
|\langle T_{r,m,n}f,g\rangle |\lesssim\ \frac{m^2}{2^m}\mathfrak{H}^r \|f\|_p\|g\|_{p'}.
\end{equation}
%Dual arguments implies that for all $g\in L^{p'}(\R^2)$,
Remember the definition  in (\ref{df2}),
Fourier inverse transform and H\"{o}lder's inequality give
\begin{equation*}
\begin{aligned}
 &\ {\rm LHS\  of}\  (\ref{A41})
 =\
 2^{-m}|\langle\sum_{j\in S_l}\sum_{k\in\Z}
  \widehat{\Phi_r}(\frac{u(x)}{2^{j-k+n}})\\
  &\ \ \  \ \times 
  \int (f*_x\Phi_{-r-1,jl+m}*_y\Phi_{r,k})
  (x-P(2^{-j}t),y-u(x)2^{-j}t) \rho_{j,r}(t)dt,g*_y\Psi_{k}
  \rangle |\\
  \le&\ 2^{-m} \|I\|_{L^{p'}_{x,y}(l^2_{j,k})}
  \|II\|_{L^{p}_{x,y}(l^2_{j,k})},
\end{aligned}
\end{equation*}
where $\{W_{j,k}\}_{j,k}\in l^2_{j,k}$ means $(\sum_{j,k}|W_{j,k}|^2)^{1/2}\lesssim1$,
$I:=\sum_{|j'-j|\le 1}|\widehat{\Phi}(\frac{u(x)}{2^{j'-k+n}})||g*_y\Psi_{k}|$
and
$$II:=\widehat{\Phi_r}(\frac{u(x)}{2^{j-k+n}})\int (f*_x\Phi_{-r-1,jl+m}*_y\Phi_{r,k})
  (x-P(2^{-j}t),y-u(x)2^{-j}t) \rho_{j,r}(t)dt.
 $$
 In what follows, we assume $t\in (1/9,9)$.  By the change of variable $2^{jl}P(2^{-j}t)\to \tau$ (since $j\in S_l$),
we have
$$
II:=\widehat{\Phi_r}(\frac{u(x)}{2^{j-k+n}})\int (f*_x\Phi_{-r-1,jl+m}*_y\Phi_{r,k})
  (x-2^{-jl}\tau),y-u(x)2^{-j}t(\tau)) \rho_{j,r}(t(\tau))t'(\tau)d\tau,
$$
where the function $t(\tau)$ satisfying $t'(\tau)\thicksim1$ is the inverse function of 
$\tau(t)=2^{jl}P(2^{-j}t)$.
%We also have $|\tau(t)-t^l|\le \frac{1}{100 N}$. 
Thanks to $\sum_{j\in S_l}(\sum_{|j'-j|\le 1}|\widehat{\Phi}(\frac{u(x)}{2^{j'-k+n}})|)^2\lesssim 1$, we have by Littlewoode-Paley theorem
 $$\|I\|_{L^{p'}_{x,y}(l^2_{j,k})}
 \lesssim\ \|(\sum_{k\in\Z}|g*_y\Psi_{k}|^2)^\frac{1}{2}\|_{p'}
 \lesssim\ \|g\|_{p'}.$$
  For the estimate of $\|II\|_{L^{p}_{x,y}(l^2_{j,k})}$,
since
 $|\int_\R\sup_{j\in S_l} \rho_{j,r}(t(\tau))t'(\tau)d\tau|\lesssim (2l)^{r+1},$
 it suffices to show that for all $\tau_0\in(1/9,9)$,
 \begin{equation}\label{a55}
 \Big\|\|\widehat{\Phi_r}(\frac{u(x)}{2^{j-k+n}}) (f*_x\Phi_{-r-1,jl+m}*_y\Phi_{r,k})
  (x-2^{-jl}\tau_0^l),y-u(x)2^{-j}t(\tau_0^l))\|_{l^2_{j,k}}\Big\|_p\lesssim\ 10^{10r}mn\|f\|_p.
 \end{equation}
 %We remark that  $t\in  (1/9,9)$  can be replaced by any bounded interval $I\subset \R^+$ or $\R^-$ with $|I|\lesssim 1$.
 %We need a lemma as follows.
 \begin{lemma}\label{l900}
  Fix $x$. For all $\{h_j(y)\}_{j\in \Z}\in L^p_y(l^2_j)$, we have for all $t\in(1/9,9)$,
 \begin{equation}\label{901}
\big\|\|\widehat{\Phi_r}(\frac{u(x)}{2^{j-k+n}})
(h_j*_y \Phi_{r,k})(x,y-u(x)2^{-j}t)\|_{l^2_{j,k}}\big\|_{L^p_y}
\lesssim\ n 10^{8r}\|h_j(y)\|_{L^p_y(l^2_j)}.
 \end{equation}
 \end{lemma}
 We postpone the proof of Lemma \ref{l900} and continue the proof of (\ref{a55}). For any fixed $x$, choosing $h_j(y):=f*_x\Phi_{-r-1,jl+m}(x-2^{-jl}\tau_0^l,y)$,  thanks to (\ref{901}),  it suffices to show that LHS of (\ref{a55}) is $\lesssim$
  \begin{equation}\label{SME}
n 10^{8r} \|\Big(\sum_{j\in S_l}|(f*_x\Phi_{-r-1,jl+m})(x-2^{-jl}\tau_0^l,y)|^2\Big)^\frac{1}{2}
 \|_p,
 %\lesssim_N\ mn 10^{9r}\|f\|_p,
 \end{equation}
whose proof is based on
 \begin{equation}\label{SS1}
\begin{aligned}
&\ 
 \chi_{S_l}(j)|f*_x\Phi_{-r-1,jl+m})(x-2^{-jl}\tau_0^l,y)|\\
 \lesssim&\ 2^{3r}\Big(M^{[\tau_0^l2^{m+1}]}_1f(x,y)
  +\sum_{k\ge 0}2^{-2k} M^{[\tau_0^l2^{m-k+1}]}_1f(x,y)
 \Big),
 \end{aligned}
 \end{equation}
 in which  $M^{[\cdot]}_1$ is  the shifted maximal operator applied in the first variable, see the Appendix for its definition. We postpone the proof of (\ref{SS1}) at the end of this subsection.
 In reality, thanks to (\ref{SS1}),  utilizing the vector-valued shifted maximal estimate (\ref{sme1}) and Littlewood-Paley theorem,
   we deduce
 $$
 \begin{aligned}
   {\rm  (\ref{SME})}\lesssim\ \  &\ n  10^{9r}
    \Big\{ \|\Big(\sum_{j\in S_l}|M^{[\tau_0^l2^{m+1}]}_1(\Phi_{jl+m}*_xf)|^2\Big)^\frac{1}{2}
 \|_p\\
 &\ +\sum_{k\ge 0}2^{-2k} \|\Big(\sum_{j\in S_l}|M^{[\tau_0^l2^{m-k+1}]}_1(\Phi_{jl+m}*_xf)|^2\Big)^\frac{1}{2}
 \|_p\Big\}\\
 \lesssim_N&\ m n  10^{9r}\|\Big(\sum_{j\in \Z}|\Phi_{jl+m}*_xf|^2\Big)^\frac{1}{2}
 \|_p
 \lesssim_N\ mn  10^{9r}\|f\|_p.
 \end{aligned}
  $$
  We end the proof of (\ref{a55}).
  \vskip.1in
Now, it remains to verify  Lemma \ref{l900} and (\ref{SS1}).
\begin{proof}[Proof of Lemma \ref{l900}]
Denote $F(s):=\{h_j(s)\}_{j\in\Z}$, and
$$\vec{K}(y,s)(\{a_j\}_{j\in\Z}):=\{a_j \widehat{\Phi_r}(\frac{u(x)}{2^{j-k+n}})  \Phi_{r,k}(y-s+u(x)2^{-j}t) \}_{j,k},$$
and
$\vec{T}(F)(y):=\int_\R \vec{K}(y,s) F(s) ds.$
In order to prove (\ref{901}), it suffcies to show
\begin{equation}\label{911}
\|\vec{T}(F)\|_{L^p_y(l^2_{j,k})}\lesssim\ n 10^{8r}\|F\|_{L^p_y(l^2_j)}.
\end{equation}
%Our strategy is to apply Lemma \ref{a.2} in the Appendix.
We first have
$$
\begin{aligned}
\|\vec{T}(F)(y)\|_{l^2_{j,k}}^2
=&\ \sum_{j,k}\widehat{\Phi_r}(\frac{u(x)}{2^{j-k+n}})^2 |\int_\R h_j(s) \Phi_{r,k}(y-s+u(x)2^{-j}t)ds |^2\\
=&\ \sum_{j,k}\widehat{\Phi_r}(\frac{u(x)}{2^{j-k+n}})^2 |
h_j*_y \Phi_{r,k}(y+u(x)2^{-j}t) |^2\\
\le&\ 4^r\sum_{j,k} |
h_j*_y \Phi_{r,k}(y+u(x)2^{-j}t) |^2.
\end{aligned}
$$
 Fubini's theorem and Littlewood-Paley theorem give
\begin{equation}\label{bb1}
\begin{aligned}
\|\vec{T}(F)(y)\|_{L^2(l^2_{j,k})}^2\le
&\ 4^r \sum_{j\in\Z}\int_\R\sum_{k\in\Z}|h_j*_y\Phi_{r,k}|^2dy
\lesssim\  16^r \|F||_{L^2(l^2(\Z))}^2.
\end{aligned}
\end{equation}
Next, we shall prove
\begin{equation}\label{bb2}
\begin{aligned}
\int_{|y-s|>2|s-z|}
\|\vec{K}(y,s)-\vec{K}(y,z)\|_{l^2(\Z)\rightarrow l^2(\Z^2)} dy\lesssim 10^{6r} n.
\end{aligned}
\end{equation}
If (\ref{bb2}) holds, together with (\ref{bb1}) and the symmetric property of $\vec{K}(y,s)$, we deduce the desired result from the application of Theorem \ref{a.2}. Thus it remains to show (\ref{bb2}).
\vskip.1in
The translation invariance of $\vec{K}(y,s)$ gives that it suffices to show the case $s=0$, i.e.,
\begin{equation}\label{end12}
\int_{|y|>2|z|}
\|\vec{K}(y,0)-\vec{K}(y,z)\|_{l^2(\Z)\rightarrow l^2(\Z^2)} dy\lesssim\  10^{6r} n.
\end{equation}
For all $a_j\in l^2(\Z)$ and $\|a_j\|_{l^2_j}\le1$, we have
\begin{equation}\label{aam}
\begin{aligned}
&\ \ \|(\vec{K}(y,0)-\vec{K}(y,z))\{a_j\}\|_{l^2(\Z^2)}\\
=&\ \|\widehat{\Phi_r}(\frac{u(x)}{2^{j-k+n}})  \big(\Phi_{r,k}(y+u(x)2^{-j}t)-
\Phi_{r,k}(y-z+u(x)2^{-j}t)\big)\{a_j\}\|_{l^2(\Z^2)}\\
=&\ \Big(\sum_{j,k}|a_j|^2
|\widehat{\Phi_r}(\frac{u(x)}{2^{j-k+n}}) |^2
|\Phi_{r,k}(y+u(x)2^{-j}t)-
\Phi_{r,k}(y-z+u(x)2^{-j}t)|^2\Big)^{1/2}\\
\le&\ \|a_j\|_{l^2_j}
\Big(\sup_j\sum_k|\widehat{\Phi_r}(\frac{u(x)}{2^{j-k+n}}) |^2
|\Phi_{r,k}(y+u(x)2^{-j}t)-
\Phi_{r,k}(y-z+u(x)2^{-j}t)|^2\Big)^{1/2}\\
\le&\ 2^{r}
\sup_j\sum_k|\widehat{\Phi}(\frac{u(x)}{2^{j-k+n}}) |
|\Phi_{r,k}(y+u(x)2^{-j}t)-
\Phi_{r,k}(y-z+u(x)2^{-j}t)|.
\end{aligned}
\end{equation}
We will discuss the relation between $|\frac{u(x)}{2^{j-k}}|(\thicksim 2^n)$ and $|z|$. We assume $z\neq0$. Define $k_{1z}\in \Z$ by
$\frac{2^{n+10}}{|z|}\in [2^{k_{1z}},2^{k_{1z}+1})$ and $k_{2z}=k_{1z}-n-10$. Rewrite the summation of $k$ in (\ref{aam}) as follows:
$$\sum_{k}\cdot=\sum_{k\ge k_{1z}}\cdot+\sum_{k< k_{2z}}\cdot+\sum_{k_{2z}\le k< k_{1z}}\cdot.$$
Thus we have
$${\rm LHS \ of\  (\ref{aam})}
\le\ 2^{r}\sup_j\big(
\sum_{k\ge k_{1z}}\cdot+\sum_{k< k_{2z}}\cdot+\sum_{k_{2z}\le k< k_{1z}}\cdot\big)=: L_{1r}(y,z)+L_{2r}(y,z)+L_{3r}(y,z).$$
So (\ref{end12}) follows from
\begin{equation}\label{end124}
\tilde{L}_\kappa:=\int_{|y|\ge 2|z|} L_{\kappa r}(y,z) dy\lesssim\ 10^{6r}n\  {\rm\  for\  all}\  \kappa=1,2,3.
\end{equation}
For $\tilde{L}_1$,
because of  ${k\ge k_{1z}}$,
we have $|2^ku(x)2^{-j}t|\le 2^{n+4}$, which yields $|2^ky|\ge 2^{k-1}|y-z|\ge 2^{k_{1z}}|z|\ge 2|2^ku(x)2^{-j}t|$. Then
 $$
 \tilde{L}_1\lesssim 2^{4r}\sum_{k\ge k_{1z}} \int_{|y|\ge 2|z|}
 \frac{2^k}{(1+2^k|y|)^3} dy
 \lesssim 2^{4r}\sum_{k\ge k_{1z}}\frac{1}{(1+2^k|z|)^2}
 \lesssim\ 2^{4r}(2^{k_{1z}}|z|)^{-2}\lesssim 2^{4r}2^{-2n}.$$
By the fundamental theorem of calculus,  we obtain
% $$
%\begin{aligned}
%% \sup_j\sum_{k\le k_{2z}}
% \widehat{\Phi_r}(\frac{u(x)}{2^{j-k+n}})
% 2^{2k}|z||\int_0^1 \Phi_{r,k}'(y-sz+u(x)2^{-j}t) ds|
% \end{aligned}
%$$
%Thus
 $$
\begin{aligned}
 \tilde{L}_2\lesssim&\ 2^{3r}\sum_{k\le k_{2z}}2^{2k}|z|\sum_{j\in\Z} \widehat{\Phi}(\frac{u(x)}{2^{j-k+n}})
\int_{|y|\ge 2|z|}\int_0^1 \frac{1}{1+|2^k(y-sz+u(x)2^{-j}t)|^2} dsdy\\
\lesssim&\ 2^{3r}\sum_{k\le k_{2z}}2^{k}|z|\sup_j
\int_0^1\int_{|y|\ge 2^{k+1}|z|}\frac{1}{1+|y-2^k(sz-u(x)2^{-j}t)|^2} dy ds\\
\lesssim&\ 2^{3r}\sum_{k\le k_{2z}}2^{k}|z|\lesssim\ 2^{3r} 2^{k_{2z}}|z|\lesssim 2^{3r}.
 \end{aligned}
$$
Finally, we estimate $\tilde{L}_3$. Denote $\diamondsuit:=\{k\in\Z:k_{1z}\ge k\ge k_{2z}\}$,
observe that
$\sharp \diamondsuit\lesssim n$, we have
$$
\begin{aligned}
\tilde{L}_3\lesssim&\ 2^{3r}\sum_{k\in\diamondsuit}
\sum_j\widehat{\Phi}(\frac{u(x)}{2^{j-k+n}})
\int_{|y|\ge 2|z|} \big(\frac{2^k}{1+|2^k(y+u(x)2^{-j}t)|^2}+
\frac{2^k}{1+|2^k(y-z+u(x)2^{-j}t)|^2}
\big)dy\\
\lesssim&\ 2^{4r}\sum_{k\in\diamondsuit}
\int_{\R}\frac{1}{1+|y|^2}dy
\lesssim\ 2^{4r} n.
\end{aligned}
$$
Combining  the above estimates of $\tilde{L}_1$, $\tilde{L}_2$ and $\tilde{L}_3$, we  complete the proof of (\ref{end124}).  This concludes the proof of Lemma \ref{l900}.
\end{proof}
\begin{proof}[Proof of (\ref{SS1})]
Making use of the prepoerty of $\Phi$, we receive
$$\chi_{S_l}(j)|(f*_x\Phi_{-r-1,jl+m})(x-2^{-jl}\tau
_0^l,y)|
\lesssim\ 2^{3r}\int|f(x-2^{-jl}\tau
_0^l-x',y)|\frac{2^{jl+m}}{1+|2^{jl+m} x'|^3}dx'.$$
The constant $2^{3r}$ comes from the pointwise  estimate of $\Phi_{-r-1,jl+m}$.
The integral on the right side is majorized  by
$$\int_{|x'|\le 2^{-jl-m}}\cdot+\sum_{k\ge 0}\int_{|x'|\in [ 2^{-jl-m+k}, 2^{-jl-m+k+1}]}\cdot=:L_1+L_2.$$
%Due to $j\in S_l$,  and $t\in [1/9,9]$,
%\begin{equation}\label{pb}
%|P(2^{-j}t)|\le 2\cdot 10^l 2^{-jl}=2\cdot 10^l2^m(2^{-jl-m}).
%\end{equation}
For $L_1$,
we have 
$$
\begin{aligned}
L_1\lesssim&\ 2^{jl+m}\int_{|x'|\le 2^{-jl-m}}|f(x-2^{-jl}\tau
_0^l-x',y)| dx'\\
\lesssim&\ 2^{jl+m}\int_{|x-2^{-jl}\tau
_0^l-z|\le 2^{-jl-m}}|f(z,y)| dz\\
\lesssim&\ M^{[\tau_0^l2^{m+1}]}_1f(x,y).
\end{aligned}
$$
%where $M^{[\cdot]}_1$ is the shifted maximal function in $x$ %variable.
%We refer the definition in the Appendix.
As for $L_2$, in a similar fashion, we obtain
$$
\begin{aligned}
L_2\lesssim&\ \sum_{k\ge 0}2^{-3k}2^{jl+m}
\int_{|x'|\in [ 2^{-jl-m+k}, 2^{-jl-m+k+1}]}|f(x-2^{-jl}\tau
_0^l-x',y)| dx'\\
\lesssim&\  \sum_{k\ge 0}2^{-2k} M^{[\tau_0^l2^{m-k+1}]}_1f(x,y).
\end{aligned}
$$
As a result, (\ref{SS1}) have been proved by combining  the estimates of  $L_1$ and $L_2$.
\end{proof}
 \subsection{The estimate of $\h_{\De_3}f$}
 Without loss of the generality, we assume $m>0$, $n\le 0$ and $m-n>100l$. We can achieve the goal by combining  the approach leading to the estimate of $\h_{\De_2}f$ and 
 certain idea in the proof of  the estimate of $\h_{\De_1}f$.   More  precisely,
  the essential  difference with the estimate of $\h_{\De_2}f$ is that we shall first  use  Taylor's expansion
$e(\eta u(x)2^{-j})=\sum_{q\ge 0}\frac{i^q}{q!}(\eta u(x)2^{-j})^q$
to the $M_j(\xi,\eta)$. The role of  Taylor's expansion  is similar to  the one in the estimate of $\h_{\De_1}f$.
\vskip.1in
Using the above procedure,  to get the desired estimate of $\h_{\De_3}f$, it suffices to show
\begin{equation}\label{ppp}
\sum_{q\ge 0}\frac{i^q}{q!}\|
\int_{\R^2}
e(\xi x+\eta y) \widehat{f}(\xi,\eta)\sum_{j\in S_l} M_{*,j,m,q}(\xi,\eta) d\xi d\eta\
\|_p\lesssim_N m^22^{-m} \|f\|_p,
\end{equation}
where 
$$M_{*,j,m,q}(\xi,\eta):=\sum_{n\le \min\{0,m-100l\}}\sum_{k\in\Z}\widehat{\Phi}(\frac{\xi}{2^{jl+m}})
\widehat{\Phi}(\frac{\eta}{2^{k}})
\widehat{\Phi}(\frac{u(x)}{2^{j-k+n}})(\eta u(x)2^{-j})^q
\int e(\xi P(2^{-j}t))t^q\rho(t)dt.$$
For $q=0$, we can absorb the summation of $n$ by  $\widehat{\Phi}(\frac{u(x)}{2^{j-k+n}})$. To prove (\ref{ppp}),  it suffices to show that for $q=0$,
\begin{equation}\label{eeq}
\|
\int_{\R^2}
e(\xi x+\eta y) \widehat{f}(\xi,\eta)\sum_{j\in S_l} \tilde{M}_{*,j,m,0}(\xi,\eta) d\xi d\eta\
\|_p\lesssim_N m^22^{-m} \|f\|_p,
\end{equation}
where 
$$\tilde{M}_{*,j,m,0}(\xi,\eta):=\sum_{k\in\Z}\widehat{\Phi}(\frac{\xi}{2^{jl+m}})
\widehat{\Phi}(\frac{\eta}{2^{k}})
\int e(\xi P(2^{-j}t))\rho(t)dt,$$
and for $q\ge 1$,
\begin{equation}\label{ppp1}
\|
\int_{\R^2}
e(\xi x+\eta y) \widehat{f}(\xi,\eta)\sum_{j\in S_l} \tilde{M}_{*,j,n,m,q}(\xi,\eta) d\xi d\eta\
\|_p\lesssim_N C^qm^22^{nq-m} \|f\|_p,
\end{equation}
where the uniform constant $C$ is independent of $m,n,q$ and 
$$\tilde{M}_{*,j,n,m,q}(\xi,\eta):=\sum_{k\in\Z}\widehat{\Phi}(\frac{\xi}{2^{jl+m}})
\widehat{\Phi}(\frac{\eta}{2^{k}})
\widehat{\Phi}(\frac{u(x)}{2^{j-k+n}})(\eta u(x)2^{-j})^q
\int e(\xi P(2^{-j}t))t^q\rho(t)dt.$$
Next, we give the sketch of the proofs of  (\ref{eeq}) and (\ref{ppp1}).
\vskip.1in
{\bf Estimate (\ref{eeq})}\ \ 
Applying integration by parts, we have
$$
\begin{aligned}
\int e(\xi P(2^{-j}t))\rho(t)dt&\ =i \int \frac{\rho'(t)}{\xi 2^{-j}P'(2^{-j}t)}e(\xi P(2^{-j}t)) dt\\
&\ =i2^{-m}(\frac{\xi}{2^{jl+m}})^{-1}\int (\frac{P'(2^{-j}t)}{2^{-j(l-1)}})^{-1} \rho'(t) e(\xi P(2^{-j}t)) dt.
\end{aligned}
$$
Arguing similarly as in the proof of  (\ref{A41}), we can get the desired estimate for $q=0$.
\vskip.1in
{\bf Estimate (\ref{ppp1})}\ \  
Using $(\eta u(x)2^{-j})^q=2^{nq}(\eta u(x)2^{-j-n})^q$, and preforming an analogous  process  as yielding (\ref{A41}), we can also achieve the desired estimate for $q\ge 1$.
\vskip.1in
\section{Proof of Theorem \ref{t2} for $\sigma=4$:\ The estimate of $\h_{\De_4}f$}
\label{hh}
In this section, we prove Theorem \ref{t2} for $\sigma=4$, which is  the core part in this paper.
\subsection{The reduction of this estimate}\label{4.1}
Without loss of generality, we assume $m\ge n$ in $\De_4$. Then we reduce the conditions of $m,n$ to  $m>0$ and $m-100l\le n\le m$, which means the summation of $n$ is finite if $m$ is fixed. Thus we can assume    $m=n$ and set
$\De_4=\{(m,n)\in\Z^2:\ m=n>0\}$
in the  following. Taking advantage of
$\rho(t)=\rho_+(t)+\rho_-(t),$
 we just need to prove the  operator with $\rho(t)$ replaced by $\rho_+(t)$ since the case on $\rho_-(t)$ can be handled  in a similar way. For  convenience, we still use $\rho(t)$ to stand for $\rho_+(t)$.
%It is enough to prove
%$$\|\h_{\De_4}f\|_p\lesssim\ \|f\|_p,$$
%where  $\h_{4\pm}f$ is defined by (\ref{aim}) for $i=4$ %with the function $\rho(t)$ in $M_j(\xi,\eta)$ replaced %5by $\rho_\pm(t)$. Since the estimates of
%$\h_{4+}f$ and $\h_{4-}f$ are similar, we only give the 5details for the former. Define $\h_{4+,m}f$ by
%\begin{equation}\label{4m}
%\h_{4+,m}f=\int_{\xi,\eta}\widehat{f}(\xi,\eta)
%(\xi x+\eta y) \sum_{j\in S_\lp} M_{j,4+,m}(\xi,\eta)d\xi %d\eta,
%\end{equation}
%where $M_{j,4+,m}(\xi,\eta)$ is defined by
%$$M_{j,4+,m}(\xi,\eta):=\sum_{k\in\Z}\widehat{\Phi}(\frac{\xi}{2^{jl+m}})
%\widehat{\Phi}(\frac{\eta}{2^{k}})
%\widehat{\Phi}(\frac{u(x)}{2^{j-k+m}})M_{j+}(\xi,\eta).$$
%Here $M_{j+}$ is given by $M_{j}$ with the function %$\rho(t)$ replaced by $\rho_+(t)$.
\vskip.1in
Recall
$\phi_{j,\xi,\eta,x}(t):=\xi P(2^{-j}t)+\eta u(x)2^{-j}t,$
%and $\mathcal{\theta}(\om)=\rho(t(\om))t'(\om)$,
then
$M_{j,\De_4}(\xi,\eta)=\sum_{m>0}M_{j,m,\De_4}(\xi,\eta),$
where
$$
M_{j,m,\De_4}(\xi,\eta):=
\sum_{k\in\Z}
\widehat{\Phi}(\frac{\xi}{2^{jl+m}})
\widehat{\Phi}(\frac{\eta}{2^{k}})
\widehat{\Phi}(\frac{u(x)}{2^{j-k+m}})\int e^{i\phi_{j,\xi,\eta,x}(t)}\rho(t)dt.$$
Direct computations lead
$$\phi_{j,\xi,\eta,x}'(t)=\xi 2^{-j}P'(2^{-j}t)+u(x)2^{-j}\eta,\ \ \phi_{j,\xi,\eta,x}''(t)=\xi 2^{-2j}P''(2^{-j}t).$$
Denote
$
G(t):=lt^{l-1}+\sum_{2=i\neq l}^Nia_i2^{j(l-i)}t^{i-1},
$
and $\bar{G}(s):=\int_0^s G(\tau)d\tau$,
then
$$\phi_{j,\xi,\eta,x}'(t):=\xi 2^{-jl}G(t)+\eta u(x)2^{-j},
\ \phi_{j,\xi,\eta,x}''(t):= \xi 2^{-jl}G'(t),\ \phi_{j,\xi,\eta,x}(t)=\xi 2^{-jl}\bar{G}(t)+\eta u(x)2^{-j}t.$$
Notice that $G(t)$ is monotone in the support of $\rho(t)$ since $j\in S_l$, and then we denote the inverse function of $G(t)$
by $f(t)$.
Here the domains of definition of $G(t)$ and $f(t)$ can be extended to large  interval like $[\tilde{K}^{-1},\tilde{K}]$ for  sufficiently large $\tilde{K}=\tilde{K}(N)$ because we can  take $\digamma_N$ large enough in the definition of $S_l$ given in  (\ref{sing}).
\vskip.1in
When $u(x)2^{-j}\eta/\xi 2^{-jl}<0$, we have $|\phi_{j,\xi,\eta,x}''(t)|\thicksim 2^m$, which implies that  $\phi_{j,\xi,\eta,x}(t)$ has a unique critical point $t_{cx}$  satisfying
%$$G(t_{cx})=-\frac{u(x)2^{-j}\eta}{\xi 2^{-jl}},\ %t_{cx}=t_{cx}(\xi,\eta):=f\big(-(u(x)2^{-j%}\eta)/\xi 2^{-jl} \big).
%$$
\begin{equation}\label{cpo}
G(t_{cx})=-\frac{u(x)2^{-j}\eta}{\xi 2^{-jl}},\ \ {\rm and}\ t_{cx}=t_{cx}(\xi,\eta):=f\big(-\frac{u(x)2^{-j}\eta)}{\xi 2^{-jl}} \big).
\end{equation}
%However,
The support of $\widehat{\Phi}$ gives that there exist two cases:
$(1)\  \frac{u(x)2^{-j}\eta}{\xi 2^{-jl}}\thicksim -1 \ {\rm or}\ (2)\ \frac{u(x)2^{-j}\eta}{\xi 2^{-jl}}\thicksim 1.$
%(Notice that we only need to replace $``\thicksim"$ by
%$``\thicksim_l"$ when we consider the general $n$ in  $[m-100l, %m)$.)
%Here can be replaced by since we have assumed $m=n$
%It suffices to consider (1). Actually,
%since  (2) yields $|\phi_{j,\xi,\eta,x}'|\gtrsim 2^m$, we can use the %same strategy handling the estimate of $\h_{\De_2}f$
%to estimate $M_j(\xi,\eta)$.  More precisely, 
To give a precise analysis, we use  $\widehat{\Phi}(z)=\widehat{\Phi}_+(z)+\widehat{\Phi}_-(z)$ to obtain
\begin{equation*}
\widehat{\Phi}(\frac{\xi}{2^{jl+m}})
\widehat{\Phi}(\frac{\eta}{2^{k}})
\widehat{\Phi}(\frac{u(x)}{2^{j-k+m}})=
\sum_{(\omega_1,\omega_2,\omega_3)\in \{+,-\}^3}
\widehat{\Phi}_{\omega_1}(\frac{\xi}{2^{jl+m}})
\widehat{\Phi}_{\omega_2}(\frac{\eta}{2^{k}})
\widehat{\Phi}_{\omega_3}(\frac{u(x)}{2^{j-k+m}}).
\end{equation*}
 If
$(\omega_1,\omega_2,\omega_3)\in \{ (-,-,+),  (-,+,-), (+,+,+),  (+,-,-)\},
$
integration by parts gives the $L^2$ estimate with bound $2^{-m}$
since  $|\phi_{j,\xi,\eta,x}'(t)|\gtrsim 2^m$. It remains to bound  
 other four cases given by
$(\omega_1,\omega_2,\omega_3)\in
\{ (-,+,+),  (-,-,-), (+,-,+),  (+,+,-)\}.$ 
 Indeed, it suffices to show the case  $(\omega_1,\omega_2,\omega_3)=(-,+,+)$ since other three cases can be treated similarly.   That is,  we assume
 \begin{equation}\label{jx}
\widehat{\Phi}(\frac{\xi}{2^{jl+m}})
\widehat{\Phi}(\frac{\eta}{2^{k}})
\widehat{\Phi}(\frac{u(x)}{2^{j-k+m}})=
\widehat{\Phi}_-(\frac{\xi}{2^{jl+m}})
\widehat{\Phi}_+(\frac{\eta}{2^{k}})
\widehat{\Phi}_+(\frac{u(x)}{2^{j-k+m}}).
\end{equation}
 In what follows, for convenience we omit the subscript $+$ and $-$ on the right side of (\ref{jx}),  and one can keep 
 $\frac{\xi}{2^{jl+m}}\sim -1,\ \ \frac{\eta}{2^{k}}\sim 1,\ \ 
 \frac{u(x)}{2^{j-k+m}}\sim 1 
 $
 in mind.
  Note 
  $\phi_{j,\xi,\eta,x}''(t)\thicksim_l -2^m<0$ at this point. 
\vskip.1in
Applying $\sum_{s\in\Z} \theta_+\big(\frac{G(t_{cx})}{2^s}\big)=1$, we have
\begin{equation}\label{op}
\int e( \phi_{j,\xi,\eta,x}(t))\rho(t)dt
=\sum_{s\in\Z} \theta_+\big(\frac{G(t_{cx})}{2^s}\big)\int e( \phi_{j,\xi,\eta,x}(t))\rho(t)dt.
\end{equation}
Thanks to the above analysis, there exists a sufficiently large  constant
$C(N)$ such that  RHS of (\ref{op}) with  $\sum_{s\in\Z} $ replaced by  $\sum_{|s|\ge C(N)}$ can be estimated by integration by parts, which yields the desired $L^2$ estimate with bound $2^{-m}$. 
Here we have used that $|s|\ge C(N)$ leads to  $G(t_{cx})\gg1$ or $0<G(t_{cx})\ll1$ which yields  $|\phi_{j,\xi,\eta,x}'(t)|\gtrsim 2^m$. 
So we call this part ``good part".
Via the method of stationary phase, denote $\lambda:=2^m$
and
$\phi_{j,\xi,\eta,x}^m(t):=2^{-m}\phi_{j,\xi,\eta,x}(t)$,
we achieve
\begin{equation}\label{Analy}
\begin{aligned}
  M_{j,\De_4}(\xi,\eta)
  %=&\ \int e(\phi_{j,\xi,\eta,x}(\om))\theta(\om)d\om\\
  =&\ \sum_{m>0,k\in\Z}
\widehat{\Phi}(\frac{\xi}{2^{jl+m}})
\widehat{\Phi}(\frac{\eta}{2^{k}})
\widehat{\Phi}(\frac{u(x)}{2^{j-k+m}})\int e(\lambda \phi_{j,\xi,\eta,x}^m(t))\rho(t)dt\\
=&\ \sum_{m>0,k\in\Z}
\widehat{\Phi}(\frac{\xi}{2^{jl+m}})
\widehat{\Phi}(\frac{\eta}{2^{k}})
\widehat{\Phi}(\frac{u(x)}{2^{j-k+m}})\Big(\ {\rm good\  part}\\
 &\ +\sum_{|s|\le C(N)} \theta_+\big(\frac{G(t_{cx})}{2^s}\big)\int e(\lambda \phi_{j,\xi,\eta,x}^m(t))\rho(t)dt\Big)\\
  =&\ \sum_{m>0,k\in\Z}
\widehat{\Phi}(\frac{\xi}{2^{jl+m}})
\widehat{\Phi}(\frac{\eta}{2^{k}})
\widehat{\Phi}(\frac{u(x)}{2^{j-k+m}})
\Big\{\ {\rm good\  part}\ +{\rm \ main\ part\ }
  + {\rm\  error}\Big\},
\end{aligned}
\end{equation}
where
$${\rm \ main\ part\ }:=\mathcal{C}\lambda^{-\frac{1}{2}}
  \frac{e(\lambda
  \phi_{j,\xi,\eta,x}^m(t_{cx}))}
  {|(\phi_{j,\xi,\eta,x}^m)''(t_{cx})|^{1/2}}
  \sum_{|s|\le C(N)} \theta_+\big(\frac{G(t_{cx})}{2^s}\big)\rho(t_{cx}),$$
$\mathcal{C}$ is a universal constant. The ``error"
%the remainder $R_{\phi,\theta}(\lambda)$ satisfies
%$$|R_{\phi,\theta}^{(r)}(\lambda)|\le %\mathrm{C_0}\lambda^{-r-\frac{1}{2}}$$
%here $\mathrm{C_0}(\lesssim1)$ depends only on the derivative of $\theta$ and $\phi_{j,\xi,\eta,x}^m$.
 admits a better decay rate $\lambda^{-3/2}$ so that we can get its desired estimate by the Cauchy-Schwarz inequality. Precisely,  we can get the  
 associated $L^2$ estimate with bound $\lambda^{-1}=2^{-m}$.
\vskip.1in
In this section, with the assumption (\ref{jx}),
we will prove the following two lemmas. For $m>0$, define
$$\h_{\De_4,m}f(x,y)=\int_{\xi,\eta}\widehat{f}(\xi,\eta)
e(\xi x+\eta y) \sum_{j\in S_l} M_{j,m,\De_4}(\xi,\eta)d\xi d\eta.$$
\begin{lemma}[$L^2$ estimate]\label{ll2}
 There exists a positive constant $\epsilon_0$ such that
\begin{equation}\label{l2}
\|\h_{\De_4,m}f\|_2\lesssim\ 2^{-\epsilon_0 m}\|f\|_2.
\end{equation}
\end{lemma}
and
\begin{lemma}[$L^p$ estimate]\label{llp}
 For all $p\in (1,\infty)$, we have
\begin{equation}\label{lp}
\|\h_{\De_4,m}f\|_p\lesssim\ m^2\|f\|_p.
\end{equation}
\end{lemma}
Interpolating (\ref{l2}) with (\ref{lp}) yields the desired $L^p$ estimate of $\h_{\De_4,m}f$ and then the $L^p$ estimate of $\h_{\De_4}f$. Therefore, we only require to verify  Lemma \ref{ll2} and Lemma \ref{llp}. Notice that we can assume $m\ge C_1(N)$ for sufficiently large $C_1(N)$. %We point out that the proof of Lemma \ref{ll2} is the main novelty in our proof.
%\vskip.2in
\subsection{Proof of Lemma \ref{ll2} ($L^2$ estimate)}
Denote $\underline{\rho}(t_{cx}):= \theta_+\big(\frac{G(t_{cx})}{2^s}\big)\rho(t_{cx})$, 
define  $\mathfrak{M}_{j,k,m,x}(\xi,\eta)$ by
$$\mathfrak{M}_{j,k,m,x}(\xi,\eta):=2^{-\frac{m}{2}}
\widehat{\Phi}(\frac{\xi}{2^{jl+m}})
\widehat{\Phi}(\frac{\eta}{2^{k}})
\widehat{\Phi}(\frac{u(x)}{2^{j-k+m}})
  \frac{e(
  \phi_{j,\xi,\eta,x}(t_{cx}))}
  {|(\phi_{j,\xi,\eta,x}^m)''(t_{cx})|^{1/2}}\underline{\rho}(t_{cx})
$$
and denote
$$T_{j,k,m}f(x,y):=\chi_{S_l}(j)\int_{\R^2}e(\xi x+\eta y)\widehat{f}(\xi,\eta) \mathfrak{M}_{j,k,m,x}(\xi,\eta)d\xi d\eta,$$
due to the analysis in (\ref{Analy}), applying the dual arguments and Littlewood-Paley theorem,
we only need to show the following lemma.
\begin{lemma}\label{l100}
There exists a positive constant $\epsilon_1$ such that
\begin{equation}\label{aim1}
\|T_{j,k,m}f\|_2\lesssim\ 2^{-\epsilon_1 m}\|f\|_2.
\end{equation}
\end{lemma}
The proof of Lemma \ref{l100} is postponed below. Assume  (\ref{aim1}) holds.
As a matter of fact, (\ref{l2}) equals that for all $g(x,y)\in L^2$ satisfying $\|g\|_2\le1$,
$
|\big\langle \sum_{j\in S_l}\sum_{k\in\Z}T_{j,k,m}f,g\big\rangle|
\lesssim\ 2^{-\epsilon_0 m}\|f\|_2.
$
Using
$$\widehat{\Phi}(\frac{\xi}{2^{jl+m}})
\widehat{\Phi}(\frac{\eta}{2^{k}})=\widehat{\Phi}(\frac{\xi}{2^{jl+m}})
\widehat{\Phi}(\frac{\eta}{2^{k}})\sum_{|j'-jl|\le 1}\sum_{|k'-k|\le1}
\widehat{\Phi}(\frac{\xi}{2^{j'+m}})
\widehat{\Phi}(\frac{\eta}{2^{k'}}),$$
denote
 \begin{equation}\label{dfn1}
 \widehat{f_{jk}}(\xi,\eta)=\sum_{|j'-jl|\le 1}\sum_{|k'-k|\le1}\widehat{\Phi}(\frac{\xi}{2^{j'+m}})
\widehat{\Phi}(\frac{\eta}{2^{k'}})\widehat{f}(\xi,\eta),
\end{equation}
 we have by H\"{o}lder's inequality
\begin{equation}\label{FT}
\begin{aligned}
|\big\langle \sum_{j,k}T_{j,k,m}f,g\big\rangle|
=&\ |\big\langle \sum_{j,k}T_{j,k,m}f,\ \sum_{|k'-k|\le 1}\widehat{\Phi}(\frac{u(x)}{2^{j-k'+m}})
(\tilde{\Phi}_k\ast_yg)\big\rangle|\\
=&\ |\big\langle \sum_{j,k}T_{j,k,m}(f_{jk}),\ \sum_{|k'-k|\le 1}\widehat{\Phi}(\frac{u(x)}{2^{j-k'+m}})
(\tilde{\Phi}_k\ast_yg)\big\rangle|\\
\lesssim&\ \big\|\big( \sum_{j,k}|T_{j,k,m}(f_{jk})|^2 \big)^{1/2}\big\|_2\ \big\|\Big[ \sum_{j,k} \sum_{|k-k'|\le 1}\widehat{\Phi}(\frac{u(x)}{2^{j-k'+m}})^2
|(\tilde{\Phi}_k\ast_yg)|^2\Big]^{1/2}\big\|_2.
\end{aligned}
\end{equation}
Utilizing
$$\sup_{k\in\Z}\sum_{j\in\Z}\sum_{|k-k'|\le 1}\widehat{\Phi}(\frac{u(x)}{2^{j-k'+m}})^2
\lesssim\ \sup_{k\in\Z}
\sum_{j\in\Z}\widehat{\Phi}(\frac{u(x)}{2^{j-k+m}})\lesssim 1,$$
 by Littlewood-Paley theorem, we obtain
$$\big\|\Big[ \sum_{j,k} \sum_{|k-k'|\le 1}\widehat{\Phi}(\frac{u(x)}{2^{j-k'+m}})^2
|\tilde{\Phi}_k\ast_yg|^2\Big]^{1/2}\big\|_2
\lesssim\ \|\big(\sum_{k\in\Z}|\tilde{\Phi}_k\ast_yg|^2\big)^{1/2}\big\|_2\lesssim\ \|g\|_2.$$
As for the first term on the RHS of (\ref{FT}), Fubini's theorem  and (\ref{aim1}) give
$$\big\|\big( \sum_{j,k}|T_{j,k,m}(f_{jk})|^2 \big)^{1/2}\big\|_2
\lesssim\ \big( \sum_{j,k}\|T_{j,k,m}(f_{jk})\big\|_2^2 \big)^{1/2}
\lesssim\ 2^{-\epsilon_0 m}\big(\sum_{j,k}\|f_{jk}\|_2^2 \big)^{1/2}\lesssim\ 2^{-\epsilon_0 m}\|f\|_2.$$
We complete the proof of Lemma \ref{ll2}  once Lemma \ref{l100} is proved.  Next, our goal is to prove Lemma \ref{l100}.
\begin{proof}[Proof of Lemma \ref{l100}]
We  first discretize  the part associated to $u(x)$ in the phase function $\phi_{j,\xi,\eta,x}(t)$.
Let $\varrho(x)$ be a non-negative smooth function which %equals one in $[-1/4,1/4]$,
 is supported in $[-3/4,3/4]$, and  satisfies $\sum_{p_1\in\Z}\varrho(z-p_1)=1$ for all $z\in\R$. Denote
\begin{equation}\label{denote}
u_{j,k,m}(x):=\frac{u(x)}{2^{j-k+m}},\ \xi_{jlm}:=\frac{\xi}{2^{jl+m}},\ \eta_k:=\frac{\eta}{2^k}.
\end{equation}
Decompose $\widehat{\Phi}(u_{j,k,m}(x))$ as follows:
\begin{equation}\label{de1}
\begin{aligned}
&\ \widehat{\Phi}(u_{j,k,m}(x))
=\sum_{p_1\thicksim 2^m}\widehat{\Phi}(u_{j,k,m}(x))
\varrho(u_{j,k,0}(x)-p_1),\ \ {\rm where}\\
 &\ \{p_1\in \Z:\ p_1\thicksim 2^m\}=\{p_1\in\Z:\ 2^{m-1}\le p_1\le 2^{m+1}\}.
\end{aligned}
\end{equation}
Denote
$\widehat{\Phi_s}(\xi)=|\xi|^s\widehat{\Phi}(\xi)$ for $s\in\R$,
write
$$
\begin{aligned}
&\ T_{j,k,m}f(x,y)=\ \chi_{S_l}(j)
2^{-\frac{m}{2}}u_{j,k,m}(x)^{-1/2}\\
&\ \times\sum_{p_1\thicksim 2^m}
\int_{\R^2}e(\xi x+\eta y)\widehat{f}(\xi,\eta) \widehat{\Phi_{-\frac{1}{2}}}(\xi_{jlm})
\widehat{\Phi}(\eta_k)
\varrho(u_{j,k,0}(x)-p_1) \frac{e(\phi_{j,\xi,\eta,x}(t_{cx}))
}{|G'(t_{cx})|^{1/2}}\underline{\rho}(t_{cx})
d\xi d\eta.
\end{aligned}
$$
Because of  the support of $\varrho$, we will use $p_1$ to approximate $u_{j,k,0}(x)$. In fact, recall $\bar{G}(s)=\int_0^s G(\tau)d\tau$ and $f=G^{-1}$, denote
$$t_{cp_1}=t_{cp_1}(\xi,\eta):=f(-\frac{p_1\eta 2^{-k}}{\xi 2^{-jl}}),\
Q(s):=sf(-s)+\bar{G}\big(f(-s)\big),\ A(s):=\frac{\xi}{2^{jl}}
Q(\frac{s\eta 2^{-k}}{\xi 2^{-jl}}),$$
we have
$$
\begin{aligned}
e\Big(\phi_{j,\xi,\eta,x}(t_{cx})\Big)=&\
e\Big(\xi 2^{-jl} Q\big(\frac{u(x)2^{-j}\eta}{\xi 2^{-jl}}\big)\Big)
=\ e(A(u_{j,k,0}(x))-A(p_1))e(A(p_1))
\end{aligned}
$$
Using the fundamental theorem of  calculus, we have
$$A(u_{j,k,0}(x))-A(p_1)
=(u_{j,k,0}(x)-p_1)\int_0^1 A'(p_1+s(u_{j,k,0}(x)-p_1)) ds,$$
which yields for $\kappa\ge 0$
$$\big(A(u_{j,k,0}(x))-A(p_1)\big)^\kappa
=\ \big(u_{j,k,0}(x)-p_1\big)^\kappa \Big(\int_0^1 A'(p_1+s(u_{j,k,0}(x)-p_1)) ds\Big)^\kappa.$$
Since $|u_{j,k,0}(x)-p_1|\lesssim1$ and $\|A'(z)\|_{L^\infty(z\thicksim 2^m)}\lesssim1$, we deduce by Taylor's expansion that
$$e(A(u_{j,k,0}(x))-A(p_1))
=\sum_{\kappa=0}^\infty \frac{i^\kappa}{\kappa!}
A_\kappa(u_{j,k,0}(x)-p_1,
\frac{\xi}{2^{jl+m}},\frac{\eta}{2^k}), $$
for  smooth functions $\{A_\kappa(\cdot,\cdot,\cdot)\}_\kappa$ satisfying there exists an absolute constant $\tilde{C_0}$ such that
$$\|A_\kappa(\cdot,\cdot,\cdot)\|_{\mathcal{C}^2([-\frac{3}{4},\frac{3}{4}]
\times (\frac{1}{2},2)^2)}\le \tilde{C_0}^\kappa,$$
where $\mathcal{C}^2$ means the standard H\"{o}lder space $\mathcal{C}^\iota$ with $\iota=2$. To get the desired result,
it suffices to show the operator
$$
\begin{aligned}
T_{j,k,m,\kappa}f(x,y):=&\chi_{S_l}(j)
2^{-\frac{m}{2}}\sum_{p_1\thicksim 2^m}\varrho(u_{j,k,0}(x)-p_1)
\int_{\R^2}e(\xi x+\eta y)\widehat{f}(\xi,\eta) \widehat{\Phi}(\xi_{jlm})
\widehat{\Phi}(\eta_k)
e\big(A(p_1)\big)\\
&\ \times\Upsilon(\xi_{jlm}, \eta_k,u_{j,k,m}(x))A_\kappa(u_{j,k,0}(x)-p_1,
\xi_{jlm}, \eta_k)
d\xi d\eta,
\end{aligned}
$$
where $\Upsilon(\cdot,\cdot,\cdot)$ is a smooth function satisfying $\|\Upsilon(\cdot,\cdot,\cdot)\|_{\mathcal{C}^2}
\lesssim1$,
satisfies there exists an absolute $\tilde{C_1}>0$ such that
$\|T_{j,k,m,\kappa}f(x,y)\|_2\lesssim\ \tilde{C_1}^\kappa2^{-\epsilon_0 m}\|f\|_2.$
Due to the properties of $A_\kappa$ and $\Upsilon$,
we only focus on $\kappa=0$ since the estimates of
the cases $\kappa\ge 1$ have similar bounds, which  do not affect the convergence of the summation of $\kappa$.
Now, we pay attention to the operator
$$
\begin{aligned}
&\ T_{j,k,m,0}f(x,y):= \mathcal{T}_{j,k,m}f(x,y)=\chi_{S_l}(j)
2^{-\frac{m}{2}}\sum_{p_1\thicksim 2^m}\varrho(u_{j,k,0}(x)-p_1)\\
&\ \times\int_{\R^2}e(\xi x+\eta y)\widehat{f_{j,k}}(\xi,\eta) \widehat{\Phi}(\xi_{jlm})
\widehat{\Phi}(\eta_k)
e\big(A(p_1)\big) \Upsilon_1(u_{j,k,0}(x)-p_1, \xi_{jlm}, \eta_k)
d\xi d\eta,
\end{aligned}
$$
for some smooth function $\Upsilon_1$ satisfying a similar $\mathcal{C}^2$ estimate  as $\Upsilon$ and $A_0$. Here we have applied  (\ref{dfn1}).
Thus it is enough to prove
\begin{equation}\label{aim2}
\|\mathcal{T}_{j,k,m}f\|_2\lesssim\ 2^{-\epsilon_0 m}\|f\|_2.
\end{equation}
Taking the Fourier transform in $y$ variable, we have
$$
\begin{aligned}
&\ \mathcal{F}_y\{\mathcal{T}_{j,k,m}f\}(x,\eta)=
\chi_{S_l}(j)
2^{-\frac{m}{2}}\sum_{p_1\thicksim 2^m}\varrho(u_{j,k,0}(x)-p_1)\\
&\ \int_{\R}e(\xi x)\widehat{f_{j,k}}(\xi,\eta) \widehat{\Phi}(\xi_{jlm})
\widehat{\Phi}(\eta_k)
e\big(A(p_1)\big) \Upsilon_1(u_{j,k,0}(x)-p_1, \xi_{jlm}, \eta_k)
d\xi.
\end{aligned}
$$
 Plancherel's identity applied in the second variable gives
that (\ref{aim2}) equals
\begin{equation}\label{aim3}
\|\mathcal{F}_y\{\mathcal{T}_{j,k,m}f\}\|_2\lesssim\ 2^{-\epsilon_0 m}\|f\|_2.
\end{equation}
To obtain (\ref{aim3}), we need further to discrete the $\xi_{jlm}$ and $\eta_k$.
Using  a variant of the previous process, that is,
$$\widehat{\Phi}(\xi_{jlm})
=\sum_{w\thicksim -2^\frac{m}{2}}\widehat{\Phi}
(\xi_{jlm})
\varrho(\xi_{jl\frac{m}{2}}-w),\ \widehat{\Phi}(\eta_k)
=\sum_{v\thicksim 2^\frac{m}{2}}\widehat{\Phi}
(\eta_k)
\varrho(\eta_{k-\frac{m}{2}}-v),
$$
where $v\thicksim 2^\frac{m}{2}$ and
$w\thicksim -2^\frac{m}{2}$ is due to the properties of $\varrho(\cdot)$ and the assumption (\ref{jx}),
  we can see $\widehat{f_{jk}}(\xi,\eta)$
is supported in
$$[(w-\frac{3}{4})2^{jl+\frac{m}{2}},(w+\frac{3}{4})2^{jl+
\frac{m}{2}}]
\times\ [(v-\frac{3}{4})2^{k-\frac{m}{2}},(v+\frac{3}{4})2^{k-
\frac{m}{2}}],$$
and
can be extended to be a periodic function whose periods are 
$\frac{3}{2}2^{jl+\frac{m}{2}}$ for the first variable and $\frac{3}{2}2^{k-\frac{m}{2}}$ for the second variable. 
For convenience, in what follows we assume its periods are 
$2^{jl+\frac{m}{2}}$ for the first variable and $2^{k-\frac{m}{2}}$ for the second variable. 
Recall the notations (\ref{denote}), we have
$$
\begin{aligned}
 &\ \  \widehat{f_{jk}}(\xi,\eta)
 \varrho(\xi_{jl\frac{m}{2}}-w)
 \varrho(\eta_{k-\frac{m}{2}}-v)
  =\ \sum_{(s,l_1)\in\Z^2} a_{s,l_1}^{w,v}
  e^{il_1\xi_{jl\frac{m}{2}}}
  e^{is\eta_{k-\frac{m}{2}}}.
\end{aligned}
$$
Denote
$$\phi_{l_1,s}^{w,v}(\xi,\eta):=
\varrho(\xi_{jl\frac{m}{2}}-w)
  \varrho(\eta_{k-\frac{m}{2}}-v)
  e^{il_1\xi_{jl\frac{m}{2}}}
  e^{is\eta_{k-\frac{m}{2}}},
  $$
then
\begin{equation}\label{w1}
a_{s,l_1}^{w,v}:=\langle \widehat{f_{jk}},\phi_{l_1,s}^{w,v}\rangle_{{\L}^2},
\end{equation}
  where $\langle h,g\rangle_{{\L}^2}$ is defined by
  \begin{equation}\label{df}
\langle h,g\rangle_{{\L}^2}:=2^{-(jl+k)}\langle h,g\rangle=2^{-(jl+k)}\int_{[0,2^{jl+\frac{m}{2}}]\times
[0,2^{k-\frac{m}{2}}]} h(\xi,\eta) \bar{g}(\xi,\eta) d\xi d\eta.
\end{equation}
%Denote
%$$\phi_{l_1,s}^{w,v}(\xi,\eta):=\widehat{\varrho}(\frac{\xi}{2^{jl+\frac{m}{2}}}-w)
%  \widehat{\varrho}(\frac{\eta}{2^{k-\frac{m}{2}}}-v)
%  e^{il_1\frac{\xi}{2^{jl+\frac{m}{2}-1}}}
%  e^{is\frac{\eta}{2^{k-\frac{m}{2}-1}}},
%  $$
 Thus we have
  $$
  \begin{aligned}
  \widehat{f_{jk}}(\xi,\eta)=&\
  \sum_{w\thicksim -2^\frac{m}{2},
  v\thicksim 2^\frac{m}{2}}  \sum_{s,l_1} a_{s,l_1}^{w,v} e^{il_1\xi_{jl\frac{m}{2}}}
  e^{is\eta_{k-\frac{m}{2}}},
  \end{aligned}
  $$
  which, with the support of $\varrho(\cdot)$ and the definition of $a_{s,l_1}^{w,v}$ (\ref{w1}), yields
  $$\|\widehat{f_{jk}}\|_2^2
  =2^{jl+k}\sum_{w,v,s,l_1}|a_{s,l_1}^{w,v}|^2.$$
  So we rewrite $\mathcal{F}_y\{\mathcal{T}_{j,k,m}f\}$ as
  $$
\begin{aligned}
\mathcal{F}_y\{\mathcal{T}_{j,k,m}f\}(x,\eta):=
&\chi_{S_l}(j)
2^{-\frac{m}{2}}\sum_{p_1\thicksim 2^m}
\sum_{w,v,l_1,s}  \langle \widehat{f_{j,k}},\phi_{l_1,s}^{w,v} \rangle_{{\L}^2}
A_{l_1,s}^{w,v,p_1}(x,\eta),
  \end{aligned}
  $$
  where
$$
  \begin{aligned}
  A_{l_1,s}^{w,v,p_1}(x,\eta)&:=
\int_{\R}e(\xi x) e^{il_1\xi_{jl\frac{m}{2}}}
  e^{is\eta_{k-\frac{m}{2}}}e\big(A(p_1)\big)\varrho(\xi_{jl\frac{m}{2}}-w)
\varrho(\eta_{k-\frac{m}{2}}-v) \Upsilon_1(u_{j,k,0}(x)-p_1, \xi_{jlm}, \eta_k)
d\xi.
\end{aligned}
$$
To obtain the $L^2$ estimate of $\mathcal{F}_y\{\mathcal{T}_{j,k,m}f\}(x,\eta)$, we first consider the estimate of
$$
\begin{aligned}
&\ \ \|\sum_{w,v,l_1,s}  \langle \widehat{f_{j,k}},\phi_{l_1,s}^{w,v} \rangle_{{\L}^2}
A_{l_1,s}^{w,v,p_1}(x,\eta)\|_{L^2_{x,\eta}}^2\\
=&\  \sum_{w,v,l_1,s}\sum_{w',v',l_1',s'}\int_x\int_{\eta}  \langle \widehat{f_{j,k}},\phi_{l_1,s}^{w,v} \rangle_{{\L}^2}
A_{l_1,s}^{w,v,p_1}(x,\eta)\overline{\langle \widehat{f_{j,k}},\phi_{l_1',s'}^{w',v'} \rangle_{{\L}^2}
A_{l_1',s'}^{w',v',p_1}(x,\eta)}
d\eta dx\\
=&\ \sum_{w,v,l_1,s}
\sum_{w',l_1',s'}
\langle \widehat{f_{j,k}},\phi_{l_1,s}^{w,v} \rangle_{{\L}^2}
\overline{\langle \widehat{f_{j,k}},\phi_{l_1',s'}^{w',v} \rangle_{{\L}^2}}\int_x \Omega (x) dx,
\end{aligned}
$$
where
$$\Omega(x)=\int_\eta A_{l_1,s}^{w,v,p_1}(x,\eta)
\overline{A_{l_1',s'}^{w',v,p_1}(x,\eta)}\ d\eta $$
and we have used the orthogonality of the support  of $\Upsilon_1$ to absorb  $\sum_{v'\thicksim 2^\frac{m}{2}}$.
\vskip.1in
Denote $\xi'_{jlm}=\frac{\xi'}{2^{jl+m}}$, %$\eta'_k=\frac{\eta'}{2^k}$ and
$$
B^{w,v}(x,p_1,\xi,\eta)=
\Upsilon_1(u_{j,k,0}(x)-p_1,\frac{\xi}{2^\frac{m}{2}}, \frac{\eta}{2^\frac{m}{2}})\varrho(\xi_{jl\frac{m}{2}}-w)
\varrho(\eta_{k-\frac{m}{2}}-v),
$$
then rewrite $\Omega(x)$ as
$$
\begin{aligned}
\Omega(x)=&\int_\eta\int_{\xi,\xi'} e((\xi-\xi')x) e^{il_1\xi_{jl\frac{m}{2}}-il_1'\xi'_{jl\frac{m}{2}}}e^{is\eta_{k-\frac{m}{2}}-is'
\eta_{k-\frac{m}{2}}}  e^{i\Big(\xi_{jl0} Q(\frac{p_12^{-m}\eta_k}{\xi_{jlm}
  })-\xi'_{jl0}
  Q(\frac{p_12^{-m}\eta_k}{\xi'_{jlm}}
)\Big)}\\
  &\ \
    B^{w,v}(x,p_1,\xi_{jl\frac{m}{2}},
  \eta_{k-\frac{m}{2}})
  \overline{B^{w',v}(x,p_1,\xi'_{jl\frac{m}{2}},
  \eta_{k-\frac{m}{2}})} \ d\xi d\xi'\ d\eta
\end{aligned}
$$
Changing the variable $\xi\rightarrow 2^{jl+\frac{m}{2}}\xi$,
 $\eta\rightarrow 2^{k-\frac{m}{2}}\eta$,
 $\xi'\rightarrow 2^{jl+\frac{m}{2}}\xi'$, rewrite $\Omega(x)$ as
 $$
\begin{aligned}
\Omega(x)=&\ 2^{2jl+\frac{m}{2}+k}\int_\eta\int_{\xi,\xi'} e^{i2^{jl+\frac{m}{2}}(\xi-\xi')x} e^{il_1\xi
+is\eta}
 e^{-il_1'\xi'-is'\eta}   e^{i\Big(\xi2^\frac{m}{2}
Q(\frac{p_1\eta_m}{\xi})-\xi'2^\frac{m}{2}
  Q(\frac{p_1\eta_m}{\xi'})\Big)}\\
  &\ \
    B^{w,v}(x,p_1,\xi,
  \eta)
  \overline{B^{w,v}(x,p_1,\xi',
  \eta)} d\xi d\xi' d\eta.
\end{aligned}
$$
%Here we have used the denotations (\ref{denote}).
Next, we give the estimate of $|\Omega(x)|$. Denote
$$\Theta_{v,w,l_1}(x):=2^\frac{m}{2} (\bar{G}\circ f)(-\frac{p_1v
}{{2^m}w})+2^{jl+\frac{m}{2}}x+l_1,$$
$$\Theta_{v,w,w',s,s'}(x):=
\frac{p_1}{2^\frac{m}{2}}\big(f(-\frac{p_1v
}{{2^m}w})
-f(-\frac{p_1v
}{{2^m}w'})\big)+s-s',$$
we have the following lemma providing useful pointwise estimate of $\Omega(x)$.
\begin{lemma}\label{l6.1}
\begin{equation}\label{eo}
\begin{aligned}
|\Omega(x)|
\lesssim&\ 2^{2jl+\frac{m}{2}+k}
\varrho(\frac{u(x)}{2^{j-k}}-p_1)
\frac{\min\{1,(\Theta_{v,w,l_1}(x))^{-2}\}
\min\{1,(\Theta_{v,w',l_1'}(x))^{-2}\}}
{1+(\Theta_{v,w,w',s,s'}(x))^2}.
\end{aligned}
\end{equation}
\end{lemma}
The proof is postponed in the section \ref{slemma}. We continue the proof of
(\ref{aim1}). By using the support  of $\Upsilon_1(u_{j,k,0}(x)-p_1,\xi_{jlm}, \eta_k)$ and
the mean value theorem, we have
  $$
\begin{aligned}
&\ \ \|\mathcal{F}_y\{
\mathcal{T}_{j,k,m}f\}(x,\eta)\|_{2}^2\\
\lesssim&\ \chi_{S_l}(j)
2^{-m}\sum_{p_1\thicksim 2^m}
\sum_{w,v,l_1,s}
\sum_{w',l_1',s'}
\langle \widehat{f_{j,k}},\phi_{l_1,s}^{w,v} \rangle_{{\L}^2}
\overline{\langle \widehat{f_{j,k}},\phi_{l_1',s'}^{w',v} \rangle_{{\L}^2}}|\int_\R \Omega (x) dx|\\
\lesssim&\ 2^{2jl+\frac{m}{2}+k}\chi_{S_l}(j)
2^{-m}
\sum_{w,v,l_1,s}
\sum_{w',l_1',s'}
\langle \widehat{f_{j,k}},\phi_{l_1,s}^{w,v} \rangle_{{\L}^2}
\overline{\langle \widehat{f_{j,k}},\phi_{l_1',s'}^{w',v} \rangle_{{\L}^2}}\\
&\ |\int_\R
\varrho(\frac{u(x)}{2^{j-k+m}})
\frac{1}{1+(\tilde{\Theta}_{v,w,l_1}(x))^2}
\frac{1}{1+(\tilde{\Theta}_{v,w',l_1'}(x))^2}
\frac{1}{1+(\tilde{\Theta}_{v,w,w',s,s'}(x))^2} dx|,
\end{aligned}
$$
where $\tilde{\Theta}_{v,w,l_1}(x)$ and $\tilde{\Theta}_{v,w,w',s,s'}(x)$ are the previous $\Theta_{v,w,l_1}(x)$ and $\Theta_{v,w,w',s,s'}(x)$ with $p_1$ replaced by
$\frac{u(x)}{2^{j-k}}$, respectively.
This process is valid since the scale $2^m$ of $p_1$ is larger than the scale $2^{m/2}$ of $w$ or $v$.
Recall (\ref{df}),
denote
$(\sum_{s}\langle \widehat{f},\phi_{l_1,s}^{w,v}\rangle^2)^\frac{1}{2}
=C_{l_1}^{w,v},\ (\sum_{s'}\langle \widehat{f},\phi_{l_1',s'}^{w',v}\rangle^2)^\frac{1}{2}
=C_{l_1'}^{w',v},$
we further obtain
  $$
\begin{aligned}
\|\mathcal{F}_y\{
\mathcal{T}_{j,k,m}f\}(x,\eta)\|_{2}^2
\lesssim&\ 2^{-\frac{m}{2}}\chi_{S_l}(j)
\sum_{w,v,l_1}
\sum_{w',l_1'}2^{-k}\int_x \varrho(\frac{u(x)}{2^{j-k+m}})
\frac{1}{1+(\tilde{\Theta}_{v,w,l_1}(x))^2}
\frac{1}{1+(\tilde{\Theta}_{v,w',l_1'}(x))^2}\\
&\ \Big\{\sum_{s,s'} \frac{1}{1+(\tilde{\Theta}_{v,w,w',s,s'}(x))^2}
|\langle \widehat{f_{j,k}},\phi_{l_1,s}^{w,v} \rangle|
|\langle \widehat{f_{j,k}},\phi_{l_1',s'}^{w',v'} \rangle|\Big\}dx.
\end{aligned}
$$
The summation in brace is bounded by
$$
\begin{aligned}
 &\ (\sum_{s}|\langle \widehat{f_{j,k}},\phi_{l_1,s}^{w,v} \rangle|^2)^\frac{1}{2} \Big(\sum_s\big(\sum_{s_1} \frac{1}{1+(\tilde{\Theta}_{v,w,w',s,s'}(x))^2}
|\langle \widehat{f_{j,k}},\phi_{l_1',s'}^{w',v'} \rangle|\big)^2\Big)^\frac{1}{2}
\lesssim\ C_{l_1}^{w,v} C_{l_1'}^{w',v},
\end{aligned}
$$
where we have used Young's inequality for series.
Consequently,
$$
\begin{aligned}
&\ \|\mathcal{F}_y\{
\mathcal{T}_{j,k,m,0}f\}(x,\eta)\|_{2}^2\\
\lesssim&\ 2^{-m/2-k}\chi_{S_l}(j)\sum_{w,v,l_1}\sum_{w',l'_1}
\int_x \varrho(\frac{u(x)}{2^{j-k+m}})
\frac{1}{1+(\tilde{\Theta}_{v,w,l_1}(x))^2}
\frac{1}{1+(\tilde{\Theta}_{v,w',l_1'}(x))^2}C_{l_1}^{w,v} C_{l_1'}^{w',v}
dx\\
\lesssim&\ 2^{-m/2-k}\chi_{S_l}(j)\sum_{v\thicksim 2^\frac{m}{2}}\int_x \varrho(\frac{u(x)}{2^{j-k+m}})
(\sum_{w,l_1}\frac{1}{1+(\tilde{\Theta}_{v,w,l_1}(x))^2}
C_{l_1}^{w,v} )^2
dx.
\end{aligned}
$$
Denote
$\tilde{C}_{l_1}^{w,v}=2^{-\frac{jl+k}{2}}C_{l_1}^{w,v} $
then
$
\|\mathcal{F}_y\{
\mathcal{T}_{j,k,m}f\}(x,\eta)\|_{2}^2
\lesssim\ \sum_{v\thicksim 2^\frac{m}{2}}\Re_{1v}^2,
$
where
\begin{equation}\label{key1}
\Re_{1v}:=\Big(2^{jl-\frac{m}{2}}
 \int_x \chi_{S_l}(j) \varrho(\frac{u(x)}{2^{j-k+m}})
(\sum_{w,l_1}\frac{\tilde{C}_{l_1}^{w,v}}{1+(\tilde{\Theta}_{v,w,l_1}(x))^2}
)^2
dx\Big)^\frac{1}{2}.
\end{equation}
To get the desired result, we  need a lemma
whose proof is
based on Lemma \ref{p1} and 
 postponed at the end of this subsection.
\begin{lemma}\label{cha}
There exists a constant $\epsilon_2>0$ such that
\begin{equation}\label{chacha}
\Re_{1v}^2\lesssim\ 2^{-\epsilon_2 m}\sum_{w, l_1}|\tilde{C}_{l_1}^{w,v}|^2.
\end{equation}
\end{lemma}
Thanks to this lemma, we  complete the proof of Lemma \ref{l100} by Plancherel's identity and setting $\epsilon_1=\epsilon_2$.
\end{proof}

\begin{proof}[Proof of Lemma \ref{cha}]
Recall
$\tilde{\Theta}_{v,w,l_1}(x):=2^\frac{m}{2}(\bar{G}\circ f)(-\frac{u_{j,k,m}(x) v}{w})+2^{jl+\frac{m}{2}}x+l_1.$
Due to the definitions of $\bar{G}$ and $f$, it is not difficult to see that there exist two  constants $c_3>0$ and $C_3>0$ such that
\begin{equation}\label{decay1}
c_3\le (\bar{G}\circ f)(-\frac{u_{j,k,m}(x) v}{w})\le C_3.
\end{equation}
Denote $I_{j\kk}=[\kk-\frac{1}{2},\kk+\frac{1}{2}]2^{-jl}$, then $\R=\sum_{\kk\in\Z}I_{j\kk}$.
Then we split  $\sum_{l_1\in\Z}$ in (\ref{key1}) into three parts, and further bound  $\Re_{1v}^2$ by an absolute  constant times the summation of
$$\Re_{1v1}^2=2^{jl-\frac{m}{2}}
  \sum_{\kk\in\Z}\int_{I_{j\kk}}  \chi_{S_l}(j) \varrho(\frac{u(x)}{2^{j-k+m}})
(\sum_{w\thicksim -2^\frac{m}{2}}\sum_{l_1<(-x2^{jl}-C_3) 2^\frac{m}{2}}
\frac{\tilde{C}_{l_1}^{w,v}}{
1+(\tilde{\Theta}_{v,w,l_1}(x))^2}
)^2
dx,$$
$$\Re_{1v2}^2=2^{jl-\frac{m}{2}}
 \sum_{\kk\in\Z}\int_{ I_{j\kk}} \chi_{S_l}(j) \varrho(\frac{u(x)}{2^{j-k+m}})
(\sum_{w\thicksim -2^\frac{m}{2}}\sum_{l_1>(-x2^{jl} -c_3) 2^\frac{m}{2}}
\frac{\tilde{C}_{l_1}^{w,v}}{
1+(\tilde{\Theta}_{v,w,l_1}(x))^2}
)^2
dx,$$
and
$$\Re_{1v3}^2=2^{jl-\frac{m}{2}}
 \sum_{\kk\in\Z}\int_{ I_{j\kk}} \chi_{S_l}(j) \varrho(\frac{u(x)}{2^{j-k+m}})
(\sum_{w\thicksim -2^\frac{m}{2}}\sum_{-C_32^\frac{m}{2}\le l_1+2^{jl+\frac{m}{2}}x\le -c_3 2^\frac{m}{2}}
\frac{\tilde{C}_{l_1}^{w,v}}{
1+(\tilde{\Theta}_{v,w,l_1}(x))^2}
)^2
dx.$$
%Note that for $\Re_{1v1}^2$ and $\Re_{1v2}^2$ we can exploit  %the rapid decay of %$\frac{1}{1+(\tilde{\Theta}_{v,w,l_1}(x))^2}$ by %(\ref{decay1}), i.e.,
For $\Re_{1v1}^2$, changing the variable
$x\rightarrow 2^{-jl}(\kk+x)=:y_x$
 gives
$$
\begin{aligned}
\Re_{1v1}^2=&\ 2^{-\frac{m}{2}}
\int_{-\frac{1}{2}}^\frac{1}{2} \chi_{S_l}(j)\varrho(\frac{u(y_x)}{2^{j-k+m}})
\underbrace{\sum_{\kk\in\Z}(\sum_{w\thicksim -2^\frac{m}{2}}\sum_{l_1<(-\kk-x-C_3) 2^\frac{m}{2}}
\frac{\tilde{C}_{l_1}^{w,v}}{
1+(\tilde{\Theta}_{v,w,l_1}(y_x))^2}
)^2}_{=:\Xi_0}
dx.
\end{aligned}
$$
Let us denote
$B_{x,k'}:=[-k'-1-x-C_3,-k'-x-C_3]
2^\frac{m}{2},$
so $\sharp\{l\in\Z:\  l\in B_{x,k'}\}\lesssim\ 2^\frac{m}{2}$,
which implies
$$
\begin{aligned}
  \Xi_0=&\ \chi_{S_l}(j)\sum_{\kk\in\Z}(\sum_{w\thicksim -2^\frac{m}{2}}\sum_{k'\ge \kk}\sum_{l_1\in B_{x,k'}}
\frac{\tilde{C}_{l_1}^{w,v}}{
1+(\tilde{\Theta}_{v,w,l_1}(y_x))^2}
)^2
\lesssim\ \sum_{\kk\in\Z}(\sum_{k'\ge \kk}
\frac{\sum_{w\thicksim -2^\frac{m}{2}}\sum_{l_1\in B_{x,k'}}\tilde{C}_{l_1}^{w,v}}{
1+2^m|k'-\kk|^2}
)^2\\
\lesssim&\ \sum_{\kk\in\Z}(\sum_{w\thicksim -2^\frac{m}{2}}\sum_{l_1\in B_{x,\kk}}\tilde{C}_{l_1}^{w,v})^2
\big(\sum_{\kk\ge 0}\frac{1}{1+2^m|\kk|^2}\big)^2
\lesssim\  \sum_{w, l_1}|\tilde{C}_{l_1}^{w,v}|^2,
\end{aligned}
$$
where
we have used Young's inequality for series for the second inequality.
Inserting this estimate into the above integral of $\Re_{1v1}^2$ yields
$\Re_{1v1}^2\lesssim\ 2^{-m/2}\sum_{w, l_1}|\tilde{C}_{l_1}^{w,v}|^2. $
Similarly, we have
$\Re_{1v2}^2\lesssim\ 2^{-m/2}\sum_{w, l_1}|\tilde{C}_{l_1}^{w,v}|^2. $
For the last term $\Re_{1v3}^2$, by the mean value theorem and change the variable
$x\rightarrow 2^{-jl}(\kk+x):=y_x$ again, we obtain
$$\begin{aligned}
\Re_{1v3}^2
%\lesssim&\ 2^{jl-\frac{m}{2}}
% \sum_{\kk\in\Z}\int_{I_{j\kk}} \chi_{S_l}(j) \varrho(\frac{u(x)}%{2^{j-k+m}})\\
%&\ \times(\sum_{w\thicksim -2^\frac{m}{2}}\sum_{ 2^{-%%\frac{m}{2}}l_1+x+\kk\thickapprox -1 }
%\frac{\tilde{C}_{l_1}^{w,v}}{
%1+\Big(2^\frac{m}{2}\big((G\circ \bar{G}^{-1})(-2^{-\frac{m}%{2}}l_1-2^{jl}x)+\frac{vu_{j,k,m}(x)
%}{w}\big)\Big)^2}
%)^2
%dx\\
\lesssim&\ 2^{-\frac{m}{2}}
 \sum_{\kk\in\Z}\int_{-\frac{1}{2}}^\frac{1}{2} \chi_{S_l}(j) \varrho(u_{j,k,m}(y_x))\\
&\ \times(\sum_{w\thicksim -2^\frac{m}{2}}\sum_{
2^{-\frac{m}{2}}l_1+x+\kk\thickapprox-1  }
\frac{\tilde{C}_{l_1}^{w,v}}{
1+\Big(2^\frac{m}{2}\big((G\circ \bar{G}^{-1})(-2^{-\frac{m}{2}}l_1-x-\kk)+\frac{vu_{j,k,m}(x)
}{w}\big)\Big)^2}
)^2
dx.
\end{aligned}
$$
Here $2^{-\frac{m}{2}}l_1+x+\kk\thickapprox-1 $ means
$-C_3\le2^{-\frac{m}{2}}l_1+x+\kk\le -c_3$ ($c_3$ and $C_3$ are defined as in (\ref{decay1})).
Define
$$\mathfrak{R}_m^{(1)}=[-C_3-1/2,C_3+1/2]\cap 2^{-\frac{m}{2}}\Z,\ \ \mathfrak{R}_m^{(2)}=[-1,-1/2]\cap 2^{-\frac{m}{2}}\Z,$$
we now only need  to show that there exists $\delta_0>0$ such that for any measurable function $A_1(x)$ satisfying $|A_1(x)|\thicksim1$,
\begin{equation}\label{dis}
\tilde{R}\lesssim\ 2^{(1/2-\delta_0)m}\sum_{w, l_1}|\tilde{C}_{l_1}^{w,v}|^2,
\end{equation}
 where
$$\tilde{R}:=\sum_{\kk\in\Z}\int_{-\frac{1}{2}}^\frac{1}{2} \chi_{S_l}(j) \varrho(u_{j,k,m}(y_x))
(\sum_{(l_1,w)\in \mathfrak{R}_m^{(1)}\times \mathfrak{R}_m^{(2)}}
\frac{\tilde{C}_{2^\frac{m}{2}(l_1-\kk)}^{2^{\frac{m}{2}}w,v}
\chi_{\thickapprox-1}(l_1+x)}{
1+2^m\Big(A_1(x)(G\circ \bar{G}^{-1})(l_1+x)+w^{-1}\Big)^2}
)^2
dx.$$
Let $\delta$ be a small enough positive constant depending only on $N$.
Denote
$$
\mathfrak{L}_{G,\epsilon}:=
\{(l_1,w)\in \mathfrak{R}_m^{(1)}\times \mathfrak{R}_m^{(2)}:\ |\mathfrak{A}_{G,\epsilon}(l_1,w)|\le 2^{-2\delta m}\},\ \ \mathfrak{H}_{G,\epsilon}:=(\mathfrak{R}_m^{(1)}\times \mathfrak{R}_m^{(2)}) \setminus \mathfrak{L}_{G,\epsilon},$$
where
$$\mathfrak{A}_{G,\epsilon}(l_1,w)=\{x\in [-\frac{1}{2},\frac{1}{2}]:\ |A_1(x)(G\circ \bar{G}^{-1})(x+l_1)+w^{-1}|
\le 2^{-(1/2-2\delta)m}\}.$$
Thus in order to prove (\ref{dis}), it is enough to show
there exist two positive constants $\delta_0'$ and $\delta_0''$ such that
\begin{equation}\label{dis1}
\tilde{R}^L\lesssim\ 2^{(1/2-\delta_0')m}\sum_{w, l_1}|\tilde{C}_{l_1}^{w,v}|^2
\end{equation}
and
\begin{equation}\label{dis2}
\tilde{R}^H\lesssim\ 2^{(1/2-\delta_0'')m}\sum_{w, l_1}|\tilde{C}_{l_1}^{w,v}|^2,
\end{equation}
where
$$\tilde{R}^L:=\sum_{\kk\in\Z}\int_{-\frac{1}{2}}^\frac{1}{2} \chi_{S_l}(j) \varrho(u_{j,k,m}(y_x))
(\sum_{(l_1,w)\in \mathfrak{L}_{G,\epsilon}}
\frac{\tilde{C}_{2^\frac{m}{2}(l_1-\kk)}^{2^{\frac{m}{2}}w,v}
\chi_{\thickapprox-1}(l_1+x)}{
1+2^m\Big(A_1(x)(G\circ \bar{G}^{-1})(l_1+x)+w^{-1}\Big)^2}
)^2
dx,$$
$$\tilde{R}^H:=\sum_{\kk\in\Z}\int_{-\frac{1}{2}}^\frac{1}{2} \chi_{S_l}(j) \varrho(u_{j,k,m}(y_x))
(\sum_{(l_1,w)\in \mathfrak{H}_{G,\epsilon}}
\frac{\tilde{C}_{2^\frac{m}{2}(l_1-\kk)}^{2^{\frac{m}{2}}w,v}
\chi_{\thickapprox-1}(l_1+x)}{
1+2^m\Big(A_1(x)(G\circ \bar{G}^{-1})(l_1+x)+w^{-1}\Big)^2}
)^2
dx.$$
{\bf $\bullet$ The estimate of (\ref{dis1})}\ \hskip.2in
For the estimate of $\tilde{R}^L$, H\"{o}lder's inequality gives
$$
\begin{aligned}
  \tilde{R}^L
\lesssim&\ \sum_{\kk\in\Z}\int_{-\frac{1}{2}}^\frac{1}{2} \chi_{S_l}(j) \varrho(u_{j,k,m}(y_x))
\sum_{(l_1,w)\in \mathfrak{L}_{G,\epsilon}}
\frac{|\tilde{C}_{2^\frac{m}{2}(l_1-\kk)}^{2^{\frac{m}{2}}w,v}|^2
\chi_{\thickapprox-1}(l_1+x)}
{
1+2^m\Big(A_1(x)(G\circ \bar{G}^{-1})(l_1+x)+w^{-1}\Big)^2}\\
&\
\sum_{(l_1',w')\in \mathfrak{L}_{G,\epsilon}}
\frac{
\chi_{\thickapprox-1}(l_1'+x)}
{
1+2^m\Big(A_1(x)(G\circ \bar{G}^{-1})(l_1'+x)+(w')^{-1}\Big)^2}
dx,
\end{aligned}
$$
which, with application of Fubini's theorem, implies the right side equals  a constant times
$$
\begin{aligned}
&\ \sum_{\kk\in\Z}\sum_{(l_1,w)\in \mathfrak{L}_{G,\epsilon}}|\tilde{C}_{2^\frac{m}{2}(l_1-\kk)
}^{2^{\frac{m}{2}}w,v}|^2\int_{-\frac{1}{2}}^\frac{1}{2} \chi_{S_l}(j) \varrho(u_{j,k,m}(y_x))
\frac{
\chi_{\thickapprox-1}(l_1+x)}
{
1+2^m\Big(A_1(x)(G\circ \bar{G}^{-1})(l_1+x)+w^{-1}\Big)^2}\\
&\  \times
\sum_{(l_1',w')\in \mathfrak{L}_{G,\epsilon}}
\frac{
\chi_{\thickapprox-1}(l_1'+x)}
{
1+2^m\Big(A_1(x)(G\circ \bar{G}^{-1})(l_1'+x)+(w')^{-1}\Big)^2}
dx\\
=&\ \sum_{\kk\in\Z}\sum_{(l_1,w)\in \mathfrak{L}_{G,\epsilon}}|\tilde{C}_{2^\frac{m}{2}(l_1-\kk)
}^{2^{\frac{m}{2}}w,v}|^2(\int_{\mathfrak{A}_{G,\epsilon}(l_1,w)}
\cdot+\int_{[-\frac{1}{2},\frac{1}{2}]\setminus
\mathfrak{A}_{G,\epsilon}(l_1,w)}\cdot)
\end{aligned}
$$
Thanks to the definition of $\mathfrak{L}_{G,\epsilon}$,
the first term is majorized by
$$
\begin{aligned}
\lesssim&\ 2^\frac{m}{2} \sum_{\kk\in\Z}\sum_{(l_1,w)\in \mathfrak{L}_{G,\epsilon}}|\tilde{C}_{2^\frac{m}{2}(l_1-\kk)
}^{2^{\frac{m}{2}}w,v}|^2|\mathfrak{A}_{G,\epsilon}(l_1,w)|
\lesssim\ 2^{(1/2-2\delta)m}\sum_{w, l_1}|\tilde{C}_{l_1}^{w,v}|^2,
\end{aligned}
$$
where we have used the definition of $\mathfrak{L}_{G,\epsilon}$ and
\begin{equation}\label{222}
\sum_{w'\in \mathfrak{R}_m^{(2)}}\frac{\chi_{S_l}(j)}{
1+2^m\Big(A_1(x)(G\circ \bar{G}^{-1})(l_1+x)+(w')^{-1}\Big)^2}
\lesssim 1
\end{equation}
for the second inequality and the first inequality, respectively.
For the second term, $x\in [-\frac{1}{2},\frac{1}{2}]\setminus
\mathfrak{A}_{G,\epsilon}(l_1,w)$ yields
 $$1+2^m\Big(A_1(x)(G\circ \bar{G}^{-1})(l_1+x)+w^{-1}\Big)^2\ge 2^{2m\delta}.$$
 As a result,
 $$\sum_{\kk\in\Z}\sum_{(l_1,w)\in \mathfrak{L}_{G,\epsilon}}|\tilde{C}_{2^\frac{m}{2}(l_1-\kk)
}^{2^{\frac{m}{2}}w,v}|^2\int_{[-\frac{1}{2},\frac{1}{2}]\setminus
\mathfrak{A}_{G,\epsilon}(l_1,w)}\cdot
\lesssim\ 2^{(1/2-2\delta)m}\sum_{w, l_1}|\tilde{C}_{l_1}^{w,v}|^2.
$$
This completes the proof of (\ref{dis1}) by setting $\delta_0'=2\delta$.\\
{\bf $\bullet$ The estimate of (\ref{dis2})}\ \hskip.2in
For the estimate of $\tilde{R}^H$, we need Lemma \ref{p1}.
Denote
$$E_1:=\{l_1\in \mathfrak{R}_m^{(1)}:\ \exists\ w\in \mathfrak{R}_m^{(2)},\ s.t.\ (l_1,w)\in \mathfrak{H}_{G,\epsilon}\},$$
applying Lemma \ref{p1} with $B(x,t)=A_1(x)(G\circ \bar{G}^{-1})(t)$, we obtain
\begin{equation}\label{key0}
\sharp E_1\lesssim\ 2^{(\frac{1}{2}-\nu_0)m}
\end{equation}
for a small positive $\nu_0=\nu_0(N)$.
By H\"{o}lder's equality, we have
$$
\begin{aligned}
\tilde{R}^H\lesssim&\ \sum_{\kk\in\Z}\int_{-\frac{1}{2}}^\frac{1}{2} \chi_{S_l}(j) \varrho(u_{j,k,m}(y_x))\\
&\ \times
(\sum_{l_1\in E_1}\sum_{w\in \mathfrak{R}_m^{(2)}}
\frac{\tilde{C}_{2^\frac{m}{2}(l_1-\kk)}^{2^{\frac{m}{2}}w,v}
\chi_{\thickapprox-1}(l_1+x)}{
1+2^m\Big(A_1(x)(G\circ \bar{G}^{-1})(l_1+x)+w^{-1}\Big)^2}
)^2
dx\\
\lesssim&\ \sum_{\kk\in\Z}\sum_{l_1\in E_1}\sum_{w\in \mathfrak{R}_m^{(2)}}|\tilde{C}_{2^\frac{m}{2}(l_1-\kk)
}^{2^{\frac{m}{2}}w,v}|^2\int_{\frac{1}{2}}^\frac{1}{2} \chi_{S_l}(j) \varrho(u_{j,k,m}(y_x))\\
&\ \times
\Big(\sum_{l_1\in E_1}\sum_{w\in \mathfrak{R}_m^{(2)}}
\frac{
\chi_{\thickapprox-1}(l_1+x)}{
1+2^m\Big(A_1(x)(G\circ \bar{G}^{-1})(l_1+x)+w^{-1}\Big)^2}
\Big)dx,
\end{aligned}
$$
which is majorized by a constant multiplied by
$$
\begin{aligned}
 &\ \sharp E_1\sum_{\kk\in\Z}\sum_{l_1\in E_1}\sum_{w\in \mathfrak{R}_m^{(2)}}|\tilde{C}_{2^\frac{m}{2}(l_1-\kk)
}^{2^{\frac{m}{2}}w,v}|^2
\lesssim\ 2^{(1/2-2\nu)m}\sum_{w, l_1}|\tilde{C}_{l_1}^{w,v}|^2
\end{aligned}
$$
via (\ref{222}) and (\ref{key0}).
This completes the proof of (\ref{dis2}) by setting $\delta_0''=2\nu$.
\end{proof}
\subsection{Proof of Lemma \ref{llp} ($L^p$ estimate)}
In this subsection, we prove Lemma \ref{llp}. Denote
$$M_{j,k,m,\De_4}(\xi,\eta):=
\widehat{\Phi}(\frac{\xi}{2^{jl+m}})
\widehat{\Phi}(\frac{\eta}{2^{k}})
\widehat{\Phi}(\frac{u(x)}{2^{j-k+m}})\int e^{i\phi_{j,\xi,\eta,x}(t)}\rho(t)dt,$$
and
$$\bar{T}_{j,k,m}f(x,y):=\int_{\xi,\eta}\widehat{f}(\xi,\eta)
e(\xi x+\eta y) M_{j,k,m,\De_4}(\xi,\eta) d\xi d\eta,$$
then
$$\h_{\De_4,m}f(x,y)=\sum_{j\in S_l}\sum_{k\in\Z} \bar{T}_{j,k,m}f(x,y).$$
By the dual arguments, it suffices to show for all $g\in L^{p'}(\R^2)$,
\begin{equation}\label{dd2}
|\langle\sum_{j\in S_l}\sum_{k\in\Z} \bar{T}_{j,k,m}f,g \rangle|\lesssim\ m^2 \|f\|_p\|g\|_{p'}.
\end{equation}
%Using (\ref{de1}) again, we have
%$$M_{j,k,m,\De_4}(\xi,\eta):=\sum_{p_1\thicksim 2^m}
%M_{j,k,m,p_1}(\xi,\eta),\
%\bar{T}_{j,k,m}f(x,y):=\sum_{p_1\thicksim %2^m}\bar{T}_{j,k,m,p_1}f(x,y),$$
%where $M_{j,k,m,p_1}(\xi,\eta)$ and $\bar{T}_{j,k,m,p_1}$ are %defined by
%$$
%M_{j,k,m,p_1}(\xi,\eta)=
%\widehat{\Phi}(\frac{\xi}{2^{jl+m}})
%\widehat{\Phi}(\frac{\eta}{2^{k}})
%\widehat{\Phi}(\frac{u(x)}{2^{j-k+m}})\varrho %(u_{j,k,0}(x)-p_1)\int e^{i\phi_{j,\xi,\eta,x}(t)}\rho(t)dt,$$
%and
%$$\bar{T}_{j,k,m,p_1}f(x,y)=
%\int_{\xi,\eta}\widehat{f}(\xi,\eta)
%e(\xi x+\eta y) M_{j,k,m,p_1}(\xi,\eta) d\xi d\eta.$$
Rewrite $\bar{T}_{j,k,m}f(x,y)$ as 
$$
\begin{aligned}
 \bar{T}_{j,k,m}f(x,y)=&\
  \int_\R (f*_x \Phi_{jl+m}*_y\Phi_k)
  (x-P(2^{-j}t),y-u(x)2^{-j}t)\rho(t)dt.
\end{aligned}
$$
Recall the definition of $\Psi_k(\cdot)$ in (\ref{dfn1}),
  LHS of (\ref{dd2}) equals
$$
\begin{aligned}
&\ |\langle\sum_{j\in S_l}\sum_{k\in\Z} \bar{T}_{j,k,m}f,
\sum_{|j'-j|\le1}\widehat{\Phi}(\frac{u(x)}{2^{j'-k+m}})\  (g*_y {\Psi}_k) \rangle|\\
\lesssim&\
\|\Big(\sum_{j\in S_l}\sum_{k\in\Z} |\bar{T}_{j,k,m}f|^2\Big)^\frac{1}{2}\|_p
\|\Big(\sum_{j\in S_l}\sum_{k\in\Z}
\sum_{|j'-j|\le1}\widehat{\Phi}(\frac{u(x)}{2^{j'-k+m}})^2
|g*_y {\Psi}_k|^2\Big)^\frac{1}{2}\|_{p'}.
\end{aligned}
$$
Since
$\sum_{j\in S_l}\sum_{|j'-j|\le1}\widehat{\Phi}(\frac{u(x)}{2^{j'-k+m}}
)^2\lesssim1,$
as the previous process, we bound the second term  by a constant times
$\|\Big(\sum_{k\in\Z}
|g*_y \tilde{\Phi}_k|^2\Big)^\frac{1}{2}\|_{p'}$, which is $\lesssim \|g\|_{p'}.$
Using the previous way yielding (\ref{a55}) with $r=0$ and $m=n$, we deduce that the first term  is  $\lesssim m^2\|f\|_p$. This completes the proof of Lemma \ref{llp}.
\vskip.1in
\section{Proof of Lemma \ref{l6.1}}
\label{slemma}
\begin{proof}[Proof of Lemma \ref{l6.1}]
Denote
$W_{\eta,x}(\xi):=\xi2^\frac{m}{2} Q(\frac{p_1\eta_m}{\xi
})+2^{jl+\frac{m}{2}}x\xi+l_1\xi,$
$$I_x(\eta):=\int_\xi e^{iW_{\eta,x}(\xi)}   B^{w,v}(x,p_1,\xi,
  \eta) d\xi,\ \ \mathcal{I}_x(\eta):=\int_{\xi'}  e^{iW_{\eta,x}(\xi')} B^{w,v}(x,p_1,\xi',
  \eta) d\xi', $$
  then
$
\begin{aligned}
\Omega(x)=&\ 2^{2jl+\frac{m}{2}+k} \int_\eta
I_x(\eta) \overline{\mathcal{I}_x(\eta)} e^{i(s-s')\eta} d\eta.
\end{aligned}
  $
 Direct computations lead to
$$W_{\eta,x}'(\xi)=2^\frac{m}{2}
(\bar{G}\circ f)(-\frac{p_1\eta_m}{\xi})+2^{jl+\frac{m}{2}}x+l_1.
%W_{\eta,x}''(\xi)=2^\frac{m}{2}\Q'(-
%\frac{\xi}{\frac{p_1}{2^m}
%\eta})(-\frac{2^m}{p_1\eta}),
$$
  Due to the support of $ B^{w,v}(x,p_1,\xi,
  \eta)$, we have
  $|I_x(\eta)|\lesssim1$
and there exists an absolute constant $\bar{C}>0$ such that
  \begin{equation}\label{gap}
  |W_{\eta,x}'(\xi)-W_{v,x}'(w)|\le \bar{C}.
  \end{equation}
  If $|W_{v,x}'(w)|\ge\ 2\bar{C}$,  we have
  $|W_{\eta,x}'(\xi)|\ge\ \bar{C}.$
  Integrating by parts twice, we obtain
  $$
  \begin{aligned}
  I_x(\eta)=&\ \int_\xi \frac{1}{iW_{\eta,x}'(\xi)}\frac{d}{ d\xi}(e^{iW_{\eta,x}(\xi)})  B^{w,v}(x,p_1,\xi,
  \eta) d\xi\\
  =&\ -\int_\R e^{iW_{\eta,x}(\xi)} \frac{d}{d\xi}\big(\frac{B^{w,v}(x,p_1,\xi,
  \eta)}{iW_{\eta,x}'(\xi)}\big)d\xi\\
  =&\ \int_\R e^{iW_{\eta,x}(\xi)} \frac{\Xi^{w,v}(x,p_1,\xi,\eta)}{(W_{\eta,x}'(\xi))^2}d\xi.
  \end{aligned}
  $$
  %We omit the details of the second integration by parts.
Here
$\|\Xi^{w,v}(x,p_1,\xi,\cdot)\|_{\mathcal{C}^2}\lesssim1.$
Via Taylor's expansion $\frac{1}{(1-z)^2}=\sum_{k=1}^\infty kz^{k-1}$ for $|z|<1$, it follows
\begin{equation}\label{kp}
\frac{1}{\big(W_{\eta,x}'(\xi)\big)^2}
=\frac{1}{\big(W_{v,x}'(w)\big)^2}
\frac{1}{\big(1-\frac{W_{v,x}'(w)-W_{\eta,x}'(\xi)
}{W_{v,x}'(w)}\big)^2}
=\frac{1}{\big(W_{v,x}'(w)\big)^2}\sum_{k=1}^\infty k \big(\frac{W_{v,x}'(w)-W_{\eta,x}'(\xi)}{W_{v,x}'(w)}
\big)^{k-1}.
\end{equation}
Plugging (\ref{kp}) into $I_x(\eta)$ gives
$$I_x(\eta)=\frac{1}{\big(W_{v,x}'(w)\big)^2}\int_\R e^{iW_{\eta,x}(\xi)}\  \sum_{k=1}^\infty  k\  \Xi^{w,v}_k(x,p_1,\xi,\eta)d\xi,$$
where
$$\Xi^{w,v}_k(x,p_1,\xi,\eta)=\Xi^{w,v}(x,p_1,\xi,\eta)  \big(\frac{W_{v,x}'(w)-W_{\eta,x}'(\xi)}{W_{v,x}'(w)}\big)^{k-1}.$$
It suffices to give the estimate of $\Xi^{w,v}_k(x,p_1,\xi,\eta)$ for $k\ge 1$ individually  since the sufficiently large  lower bound of $|W_{v,x}'(w)|$ and (\ref{gap}) yielding   $\frac{W_{v,x}'(w)-W_{\eta,x}'(\xi)}{W_{v,x}'(w)}$  is small enough can  absorb   the summation of $k$.  We only show $k=1$ since other terms can be treated analogously. 
\vskip.1in
%In fact,  compared with $\Xi^{w,v}(x,p_1,\xi,\eta)$, $\Xi^{w,v}%_1(x,p_1,\xi,\eta)$ has an analogous  property.  
Rewrite $I_x(\eta)$ as
$$I_x(\eta)=\frac{1}{\big(W_{v,x}'(w)\big)^2}\int_\R e^{iW_{\eta,x}(\xi)} \Xi^{w,v}_1(x,p_1,\xi,\eta)d\xi$$
Following the above  arguments line be line yields
$$\mathcal{I}_x(\eta)=\frac{1}{
\big(W_{v,x}'(w')\big)^2}\int_\R e^{iW_{\eta,x}(\xi')} \Xi^{w',v}_1(x,p_1,\xi',\eta)d\xi'$$
when $|W_{v,x}'(w')|\ge\ \bar{C}$.
 Therefore, when $|W_{v,x}'(w)|\ge\ \bar{C}$ and $|W_{v,x}'(w')|\ge\ \bar{C}$ hold, we arrive at
$$
\begin{aligned}
\Omega(x)=&\ 2^{2jl+\frac{m}{2}+k}\frac{1}{\big(W_v'(w')\big)^2}\frac{1}{\big(W_v'(w)\big)^2}
\int_\eta \int_{\xi,\xi'} e^{iW_\eta(\xi)}e^{iW_\eta(\xi')} \\ &\ \Xi^{w,v}_1(x,p_1,\xi,\eta) \Xi^{w',v}_1(x,p_1,\xi',\eta)d\xi d\xi' e^{i(s-s')\eta}d\eta.
\end{aligned}
$$
Denote
$$\Xi^{w,w',v}_2(x,p_1,\xi,\xi',\eta):= \Xi^{w,v}_1(x,p_1,\xi,\eta) \Xi^{w',v}_1(x,p_1,\xi',\eta)$$
and
$$W_{\xi,\xi',s-s_1}(\eta):=
\xi2^\frac{m}{2} Q(\frac{p_1\eta_m}{\xi
})-\xi'2^\frac{m}{2} Q(\frac{p_1\eta_m}{\xi'
}),$$
we have
\begin{equation}\label{FE1}
\begin{aligned}
|\Omega(x)|\lesssim&\ 2^{2jl+\frac{m}{2}+k}\frac{1}{\big(W_v'(w')\big)^2}\frac{1}{\big(W_v'(w)\big)^2}
\int_{\xi,\xi'}|\underbrace{\int_\eta e^{iW_{\xi,\xi',s-s_1}(\eta)}
\Xi^{w,w',v}_2(x,p_1,\xi,\xi',\eta)d\eta}_{Os_1}|d\xi d\xi',
\end{aligned}
\end{equation}
where we have absorbed the linear term on $\xi$ and $\xi'$ in the phase function  by the absolute value in (\ref{FE1}).
We immediately have
$$\frac{d}{d\eta}(W_{\xi,\xi',s-s_1}(\eta))
=\frac{p_1}{2^\frac{m}{2}}f(-\frac{p_1\eta_m}{\xi})
-\frac{p_1}{2^\frac{m}{2}}f(-\frac{p_1\eta_m}{\xi'})+s-s'
=\Theta_{\eta,\xi,\xi',s,s'}(x).
$$
Along the same way yielding (\ref{FE1}), with a trivial estimate
$|Os_1|\lesssim \varrho(\xi-w)
\varrho(\xi'-w')\varrho(\frac{u(x)}{2^{j-k}}-p_1)$
and an application of the process deducing (\ref{kp}),
we have
$$|Os_1|\lesssim \frac{1}{1+|\Theta_{v,\xi,\xi',s,s'}(x)|^2}\varrho(\xi-w)
\varrho(\xi'-w')\varrho(\frac{u(x)}{2^{j-k}}-p_1).$$
Thus, when $|W_{v,x}'(w)|\ge\ C_1$ and $|W_{v,x}'(w')|\ge\ C_1$, we obtain
$$
\begin{aligned}
|\Omega(x)|\lesssim&\ 2^{2jl+\frac{m}{2}+k}
\frac{1}{\big(W_v'(w')\big)^2}\frac{1}{\big(W_v'(w)\big)^2}
\frac{1}{1+|\Theta_{v,w,w',s,s'}(x)|^2}
\varrho(\frac{u(x)}{2^{j-k}}-p_1).
\end{aligned}
$$
As for $|W_{v,x}'(w)|\le\ C_1$ and  $|W_{v,x}'(w')|\le\ C_1$,  we can get the bound with $\frac{1}{\big(W_v'(w')\big)^2}$ and $\frac{1}{\big(W_v'(w)\big)^2}$) replaced by 1, repectively. This ends the proof of Lemma \ref{l6.1}.
%$$
%begin{aligned}
%|\Omega(x)|\le&\ 2^{2jl+\frac{m}{2}+k}
%\frac{1}{1+|\Theta_{v,w,w',s,s'}(x)|^2}
%\varrho(\frac{u(x)}{2^{j-k}}-p_1).
%\end{aligned}
%$$
%Hence,
%we can deduce the desired estimate.
%$$
%begin{aligned}
%\Omega(x)|\le&\ 2^{2jl+\frac{m}{2}+k}
%\frac{1}{1+\big(W_v'(w')\big)^2}\frac{1}{1+\big(W_v'(w)\big)^2}
%\frac{1}{1+|\Theta_{v,w,w',s,s'}(x)|^2}
%\varrho(\frac{u(x)}{2^{j-k}}-p_1).
%\end{aligned}
%$$
\end{proof}
\vskip.2in
\section{The maximal operator $M^\Gamma$ and the related Carleson type operator $\mathcal{C}^\Gamma$}
\label{me}
In this section, we give a remark on the uniform $L^p$ estimates  of $M^\Gamma$ and a related Carleson type operator  $\mathcal{C}^\Gamma$ given by
 $$\mathcal{C}^\Gamma f(x)=p.v.\int_\R f(x-P(t))e^{iu(x) t}\frac{dt}{t},$$
 whose proofs  are similar to that of
$\mathcal{H}^\Gamma$.
\subsection{The uniform $L^p$ estimate of $M^\Gamma$}
Recall the definition of  $M^\Gamma$ given by 
$$M^\Gamma f(x,y)=\sup_{\epsilon>0}\frac{1}{2\epsilon}
\int_{-\epsilon}^\epsilon |f(x-P(t),y-u(x)t)|dt.$$
Utilizing
$$\frac{1}{2\epsilon}\int_{-\epsilon}^\epsilon
|f(x-P(t),y-u(x)t)|dt
=\frac{1}{2}\frac{1}{\epsilon}\int_{-\frac{\epsilon}{2}
}^{\frac{\epsilon}{2}}
|f(x-P(t),y-u(x)t)|dt+
\frac{1}{2\epsilon}\int_{\frac{\epsilon}{2}\le|t|\le \epsilon}
|f(x-P(t),y-u(x)t)|dt,$$
we obtain
$M^\Gamma f(x,y)\le \frac{1}{2}M^\Gamma f(x,y)+C{\bf M}^\Gamma f(x,y),$
where $C$ is a uniform constant and ${\bf M}^\Gamma f(x,y)$ is defined by
$$
{\bf M}^\Gamma f(x,y)=\sup_{j\in\Z}
\int
|f(x-P(t),y-u(x)t)| \theta_j(t) dt.$$
Here $\theta_j(t)=2^j\theta(2^jt)$ in which $\theta$ is defined as in section \ref{s2}.
As a result, it suffices to show that
$\|{\bf M}^\Gamma f\|_p\lesssim_N\|f\|_p,$
where the input function $f$ can be restricted  to be non-negative function. 
%In fact, we can achieve this by
 %the decomposition $f=f^+-f^-$, where
 %$f^+:=\max\{f,0\}$ and $f^-:=\max\{-f,0\}$. 
 Now, we rewrite ${\bf M}^\Gamma f(x,y)$  by
 Fourier inverse transform, that is,
${\bf M}^\Gamma f(x,y)=\sup_{j\in\Z}\h_jf(x,y),
$
where
$
\h_jf(x,y)=\int_{\xi,\eta}\widehat{f}(\xi,\eta)
e(\xi x+\eta y) M_j(\xi,\eta)d\xi d\eta$
and
$$M_j(\xi,\eta):=\int e(\phi_{j,\xi,\eta,x}(t))\theta(t)dt,\ \phi_{j,\xi,\eta,x}(t):=\xi P(2^{-j}t)+\eta u(x)2^{-j}t.$$
Following the estimate of $\mathcal{H}^\Gamma f(x,y)$ line by line can lead the desired result (In fact, the estimate of $ {\bf M}^\Gamma f(x,y)$ is easier since  it does not have a summation of $j$).
\subsection{The uniform $L^p$ estimate of $\mathcal{C}^\Gamma$}
%The Fourier transform in the $y$-variable gives that the $L^2$-boundedness of this operator $\mathcal{C}^\Gamma $ equals to the associated bound of $\h^\Gamma$.
By linearization,  $\mathcal{C}^\Gamma f$ can also defined by
$$\mathcal{C}^\Gamma f(x)=\sup_{u\in\R}|p.v.\int_\R f(x-P(t))e^{iu t}\frac{dt}{t}|,$$
which can be seen as a variant of the classical Carleson maximal operator, we refer  \cite{Car66,L19,L20,SW01}.
It follows via (\ref{de1}) that
$\mathcal{C}^\Gamma f(x):=\sum_{j\in\Z} \mathcal{C}_jf(x),$
where $\mathcal{C}_jf(x)$ is defined by
$\mathcal{C}_jf(x)=\int f(x-P(t))e^{iu(x) t}\rho_j(t) dt.$
%Here $\rho_j(t)$  is defined as in section \ref{s2}.
We now rewrite $\mathcal{C}_j f(x)$ as
 $
\mathcal{C}_jf(x)=\int_{\xi}\widehat{f}(\xi)
e(\xi x) M_j(\xi)d\xi$
where
$$M_j(\xi):=\int e(\phi_{j,\xi,x}(t))\rho(t)dt,\ \phi_{j,\xi,x}(t):=\xi P(2^{-j}t)+ u(x)2^{-j}t.$$
Using a new decomposition
$
\sum_{(m,n)\in \Z^2}\widehat{\Phi}(\frac{\xi}{2^{jl+m}})
\widehat{\Phi}(\frac{u(x)}{2^{j+n}})=1,
$
%and  without the summation of $k\in\Z$,
and then executing a similar  process as  yielding the uniform estimate of $\mathcal{H}^\Gamma$,
 we can get the uniform estimate of $\mathcal{C}^\Gamma$ as well.  Actually, without the summation of $k$ in the decomposition,
the proof is easier.
\section*{Acknowledgements}
\vskip .1in
 This work was supported by the NSF of China 11901301.
\vskip.3in
\appendix
\section{}
\label{app}
\begin{lemma}[\cite{L19}]\label{la1}
Let $n,\tilde{N},K\in \N$ with $n,K\le \tilde{N}$. Assume we are given $\{I_l\}_{l=1}^\N$ sets such that  
$I_l\in [-1/2,1/2]$ and $|I_l|\ge K^{-1}$ hold whenever $1\le l\le \tilde{N}$,
Then, if $\tilde{N}\ge 2M^nn^n$, there exists a subset $\mathrm{S}\subset \{1,\cdot\cdot\cdot,\tilde{N}\}$ such that
$\sharp \mathrm{S}=n$ and the measure of $\cap_{l\in \mathrm{S}}I_l$ is greater than $2^{-1}K^{-n}$.
\end{lemma}
Here the interval $[-1/2,1/2]$ can be replaced by $[-\tilde{\mathfrak{C}},\tilde{\mathfrak{C}}]$ for arbitrary uniform  constant $\tilde{\mathfrak{C}}\gtrsim 1$.
\vskip.1in
The shifted maximal operator $M^{[\sigma]}$ 
 is defined by
$$M^{[\sigma]}f (z):=\sup_{z\in I\subset \R}\frac{1}{|I|}
\int_{I^{(\sigma)}}|f(\xi)|d\xi,$$
where $I^{(\sigma)}$ is given by $[a+\sigma |I|,b+\sigma|I|]$ if 
$I=[a,b]$.
\vskip.1in
Next, we introduce the estimate of $M^{[n]}f$  in $L^p$ and the vector-valued estimate of $\{M^{[n]}f_k\}_k$ in $L^p(l^q)$.
\begin{lemma}[\cite{M14,S93}]\label{ssme}
Let $1<p<\infty$,
we have
\begin{equation*}\label{ssme1}
\|M^{[n]}f\|_p
\lesssim\ \log (2+|n|) \|f\|_p.
\end{equation*}
Here the constant hidden in $``\lesssim"$ is independent of $|n|$ and $f$.
\end{lemma}

\begin{lemma}[\cite{GHLJ}]\label{sme}
Let $1<p<\infty$, $1<q\le \infty$,
we have
\begin{equation}\label{sme1}
\Big\|\big(\sum_{k\in\Z}|M^{[n]}f_k|^q\big)^\frac{1}{q}\Big\|_p
\lesssim\ \log^2(2+|n|) \Big\|\big(\sum_{k\in\Z}|f_k|^q\big)^\frac{1}{q}\Big\|_p.
\end{equation}
Here the constant hidden in $``\lesssim"$ is independent of $|n|$ and $f$.
\end{lemma}
\vskip.1in
We need a special case of Theorem 1.1 in \cite{GLY}. Let
$\vec{T}$ be given as
\begin{equation}\label{111}
\vec{T}(F)(y):=\int_\R\vec{K}(y,s)(F(s))ds,
\end{equation}
where the kernel $\vec{K}$ satisfies H\"{o}rmander's conditions, i.e., there exists a positive $C_H$ such that for all $s,z\in \R$,
\begin{equation}\label{112}
\int_{|y-s|>2|s-z|}\|\vec{K}(y,s)-\vec{K}(y,z)\|_{l^2(\Z)\rightarrow l^2(\Z^2)} dy\le C_H,
\end{equation}
and for all $x,w\in\R$,
\begin{equation}\label{113}
\int_{|x-y|>2|x-w|}
\|\vec{K}(x,y)-\vec{K}(w,y)\|_{l^2(\Z)\rightarrow l^2(\Z^2)} dy\le C_H.
\end{equation}
\begin{thm}[\cite{GLY}]\label{a.2}
Assume that $\vec{T}$  defined by (\ref{111}) is a bounded linear operator from $L^r(\R,l^2(\Z))$ to $L^r(\R,l^2(\Z^2))$ for some $r\in (1,\infty)$ with norm $A_r>0$. Assume that $\vec{K}$ satisfies (\ref{112}) and (\ref{113}) for some $C_H>0$. Then $\vec{T}$ has well-defined extensions on $L^p(\R,l^2(\Z))$ for all $p\in (1,\infty)$. Moreover, whenever $1<p<\infty$, for all $F\in L^p(l^2(\Z))$,
$$\|\vec{T}(F)\|_{L^p(l^2(\Z^2))}
\le \max\{p,(p-1)^{-1}\}(C_H+A_r)\|F\|_{L^p(l^2(\Z))}.$$
\end{thm}


\vskip.1in
\begin{thebibliography}{99}

\bibitem{B13} M. Bateman, Single annulus $L^p$ estimates for Hilbert transforms along vector fields, {\it  Rev. Mat. Iberoam. \bf 29} (2013), 1021-1069.

\bibitem{BT}
 M. Bateman, C. Thiele,  $L^p$ estimates for the Hilbert transforms along a one-variable vector field, {\it  Anal. PDE \bf 6} (2013), 1577-1600.

\bibitem{B89} J. Bourgain, A remark on the maximal function associated to an analytic vector field, Analysis at Urbana, vol. I (Urbana, IL, 1986-1987), London Mathematical Society Lecture Note Series 137 (Cambridge University Press, Cambridge, 1989) 111-132.



\bibitem{Car66}
L. Carleson, On convergence and growth of partial sums of Fourier series, {\it Acta Math. \bf 116} (1966) 135-157.

\bibitem{CNSW99} M. Christ, A. Nagel, E. Stein and S. Wainger, Singular and maximal Radon transforms: analysis and geometry, {\it Ann. of Math. \bf 150} (1999) 489-577.

%\bibitem{G08}
%L. Grafakos, Classical Fourier Analysis, 2nd edn, Graduate %Texts in Mathematics, vol. 249. Springer, New York (2008)

%\bibitem{CRW98}
%A. Carbery, F. Ricci, J. Wright,  Maximal functions and Hilbert transforms associated to %polynomials, {\it  Rev. Mat. Iberoamericana \bf 14} (1998),  117-144.




\bibitem{DGTZ}
F. Di Plinio, S. Guo, C. Thiele, P. Zorin-Kranich, Square functions for bi-Lipschitz maps and directional operators, {\it J. Funct. Anal. \bf 275}  (2018) 2015-2058.

\bibitem{GLY} L. Grafakos, L. Liu, D. Yang,  Vector-valued singular integrals and maximal functions on spaces of homogeneous type, {\it Math. Scand. \bf 104} (2009),  296-310.

\bibitem{G17}
S. Guo, Hilbert transform along measurable vector fields constant on Lipschitz curves: $L^p$ boundedness, {\it Trans. Am. Math. Soc. \bf 369}  (2017) 2493-2519.


\bibitem{GHLJ}
S. Guo, J. Hickman, V. Lie, J. Roos,  Maximal operators and Hilbert transforms along variable non-flat homogeneous curves, {\it  Proc. Lond. Math. Soc. \bf 115} (2017),  177-219.

\bibitem{GRSP} S. Guo, J. Roos, A. Seeger, P.-L. Yung,  A maximal function for families of Hilbert transforms along homogeneous curves, {\it Math. Ann. \bf 377} (2020), 69-114.


\bibitem{GRSP2} S. Guo, J. Roos, A. Seeger, P.-L. Yung, Maximal functions associated with families of homogeneous curves: $L^P$ bounds for $p\le 2$, {\it Proc. Edinb. Math. Soc. \bf 63} (2020), 398-412.

%\bibitem{GPRY}
% S. Guo, L.  Pierce, J.  Roos, P.-L. Yung,  Polynomial %Carleson operators along monomial curves in the plane,  {\it %J. Geom. Anal. \bf 27} (2017),  2977-3012.


\bibitem{K07} G.A. Karagulyan,  On unboundedness of maximal operators for directional Hilbert transforms, {\it  Proc. Am. Math. Soc. \bf 135} (2007) 3133-3141.


\bibitem{LMP19}
 I. ${\L}$aba, A. Marinelli, M.  Pramanik,  On the maximal directional Hilbert transform, {\it Anal. Math. \bf 45} (2019),  535-568.

 \bibitem{LL06} M. Lacey, X. Li, Maximal theorems for the directional Hilbert transform on the plane, {\it  Trans. Amer. Math. Soc. \bf 358} (2006),  4099-4117.


\bibitem{LL10} M. Lacey, X. Li,  On a conjecture of E. M. Stein on the Hilbert transform on vector fields, {\it Mem. Amer. Math. Soc. \bf 205} (2010),  viii+72 pp.




\bibitem{LX16}
X. Li, L. Xiao, Uniform estimates for bilinear Hilbert transforms and bilinear maximal functions associated to polynomials, {\it  Amer. J. Math. \bf 138} (2016),  907-962.



\bibitem{L20} V. Lie,  The polynomial Carleson operator, {\it Ann. of Math. \bf 192} (2020), 47-163.


\bibitem{L19} V. Lie,  A unified approach to three themes in harmonic analysis (1st part), arXiv:1902.03807.



\bibitem{LY22}
N. Liu, H. Yu, Hilbert transforms along variable planar curves: Lipschitz regularity, {\it J. Funct. Anal. \bf 282} (2022),  Paper No. 109340.

\bibitem{LSY21}
N. Liu, L. Song, H. Yu, $L^p$ bounds of maximal operators along variable planar curves in the Lipschitz regularity, {\it J. Funct. Anal. \bf 280} (2021), Paper No. 108888, 40 pp.


\bibitem{MR98}
G. Marletta, F. Ricci, Two-parameter maximal functions associated with homogeneous surfaces in $\R^n$, {\it Stud. Math. \bf 130}  (1998) 53-65.


\bibitem{M14}
C. Muscalu, Calder\'{o}n commutators and the Cauchy integral on Lipschitz curves revisited I. First commutator and generalizations, {\it Rev. Mat. Iberoam. \bf 30} (2014) 727-750.





\bibitem{S93}
E. Stein, Harmonic analysis: real-variable methods, orthogonality, and oscillatory integrals. With the assistance of Timothy S. Murphy, Princeton Mathematical Series 43, Monographs in Harmonic Analysis, III (Princeton University Press, Princeton, NJ, 1993) xiv+695.



\bibitem{SS12} E. Stein, B. Street, Multi-parameter singular Radon transforms III: real analytic surfaces, {\it Adv. Math. \bf 229} (2012) 2210-2238.

\bibitem{SW01} E. Stein,  S. Wainger, Oscillatory integrals related to Carleson's theorem, {\it  Math. Res. Lett. \bf  8} (2001) 789-800.

%\bibitem{W00}T. Wolff, Local smoothing type estimates on Lp $for large $p$, {\it Geom. Funct. Anal. \bf 10} (2000) $1237-1288.
%\bibitem{Yu19} H. Yu,  Hilbert transforms along double %variable fractional monomials, {\it  Commun. Pure Appl. %Anal. \bf 18} (2019),  1433-1446.











\bibitem{Wan19} R. Wan,  $L^p$ bound for the Hilbert transform along variable non-flat curves, arXiv:2010.06920,
Math. Nachr. (2022).

%\bibitem{SS12}
%E.M. Stein, B. Street, Multi-parameter singular Radon %transforms III: real analytic surfaces,  {\it Adv. Math. %\bf 229}  (2012) 2210-2238.









\end{thebibliography}
\end{document}